\documentclass[12pt,reqno]{amsart}
\usepackage[a4paper,left=35mm,right=35mm,top=35mm,bottom=35mm,marginpar=25mm]{geometry}
\usepackage[pdfusetitle, citecolor=blue, colorlinks=true]{hyperref}
\usepackage{color}
\usepackage[dvipsnames]{xcolor}

\usepackage[T1]{fontenc}
\usepackage{lmodern}
\usepackage[english]{babel}
\usepackage{microtype}

\usepackage{amsmath,amssymb,amsfonts,amsthm}
\usepackage{mathtools,accents}
\usepackage{mathrsfs}
\usepackage{aliascnt}
\usepackage{braket}
\usepackage{bm}

\usepackage{enumerate}
\usepackage{xcolor}

\usepackage{comment}
\usepackage{amsmath, amsthm, amssymb, amsfonts, mathtools}
\usepackage{esint}
\usepackage{graphicx,epsfig,color}
\usepackage{exscale}
\newcommand{\haus}{\mathscr{H}}

\theoremstyle{plain}

\newtheorem{theorem}{Theorem}[section]

\newtheorem{lemma}[theorem]{Lemma}

\newtheorem{corollary}[theorem]{Corollary}
\newtheorem{conjecture}[theorem]{Conjecture}
\newtheorem{proposition}[theorem]{Proposition}

\theoremstyle{definition}

\newtheorem{definition}[theorem]{Definition}

\theoremstyle{remark}
\newtheorem{remark}[theorem]{Remark}

\DeclareSymbolFont{AMSb}{U}{msb}{m}{n}
\DeclareMathSymbol{\N}{\mathalpha}{AMSb}{"4E}
\DeclareMathSymbol{\R}{\mathalpha}{AMSb}{"52}
\DeclareMathSymbol{\Z}{\mathalpha}{AMSb}{"5A}
\DeclareMathSymbol{\D}{\mathalpha}{AMSb}{"44}
\DeclareMathSymbol{\s}{\mathalpha}{AMSb}{"53}

\newcommand{\X}{X}

\DeclareMathOperator{\vol}{vol}
\DeclareMathOperator{\Ch}{Ch}
\DeclareMathOperator{\lip}{Lip}

\DeclareMathOperator{\supp}{supp}

\DeclareMathOperator{\m}{m}
\DeclareMathOperator{\de}{d}
\DeclareMathOperator{\ric}{ric}
\DeclareMathOperator{\diam}{diam}

\DeclareMathOperator{\Div}{div}
\DeclareMathOperator{\Test}{Test}

\newcommand{\CD}{\mathrm{CD}}
\newcommand{\RCD}{\mathrm{RCD}}

\title[Stability of tori]{Gromov-Hausdorff stability of tori under Ricci and integral scalar curvature bounds}

\author[Honda]{Shouhei Honda}
\address[Shouhei Honda]{Graduate School of Mathematical Sciences, The University of Tokyo, Tokyo, Japan}
\email{shouhei@ms.utokyo.ac.jp}
\author[Ketterer]{Christian Ketterer}
\address[Christian Ketterer]{Mathematisches Institut, Albert-Ludwigs Universit\"at, Freiburg, Germany}
\email{christian.ketterer@math.uni-freiburg.de}
\author[Mondello]{Ilaria Mondello}
\address[Ilaria Mondello]{Université Paris Est Créteil, Laboratoire d’Analyse et Mathématiques appliqués, UMR CNRS 8050, F-94010 Créteil, France.}
\email{ilaria.mondello@u-pec.fr}
\author[Perales]{Raquel Perales}
\address[Raquel Perales]{Centro de Investigaci\'on en Matem\'aticas, 
De Jalisco s/n, Valenciana, Guanajuato, Gto. Mexico. 36023}
\email{raquel.perales@cimat.mx}
\author[Rigoni]{Chiara Rigoni}
\address[Chiara Rigoni]{Faculty of Mathematics, University of Vienna, Austria}
\email{chiara.rigoni@univie.ac.at}

\begin{document}
\begin{abstract} 
We establish a nonlinear analogue of a splitting map into a Euclidean space, as a harmonic map into a flat torus. We prove that the existence of  such a map implies Gromov-Hausdorff closeness to a flat torus in any dimension. Furthermore, Gromov-Hausdorff closeness to a flat torus  and an integral bound {on $r_M(x)$, the smallest eigenvalue of the Ricci tensor $\ric_x$ in $x$}, imply the existence of a harmonic splitting map. Combining these results with Stern's inequality, we provide a new Gromov-Hausdorff stability theorem for flat $3$-tori. 
The main tools we employ include the harmonic map heat flow, Ricci flow,  and both Ricci limits and RCD theories.

\end{abstract}
\maketitle

{
  \hypersetup{linkcolor=black}
  \tableofcontents
}

\section{Introduction}
\subsection{Related known results}

Several recent results and conjectures concern torus stability, meaning that under the appropriate assumptions on curvature and the appropriate notions of distance, a torus must be close to a (possibly flat and collapsed) torus. Very recently, Bru\`e, Naber and Semola \cite{BrueSemolaNaber_torus} showed that under an arbitrary {lower bound on the sectional curvature}, a sequence of $n$-tori must converge in the Gromov-Hausdorff (GH) topology to a $k$-torus, where $0\leq k \leq n$ and this $k$-torus is not necessarily flat. This was conjectured by Petrunin, see \cite{Zamora}. In dimension 3, the sectional curvature bound can be replaced by a Ricci lower bound \cite{BrueSemolaNaber_torus}. When assuming an almost non-negative Ricci lower bound on a sequence of $n$-tori, a result of Cheeger and Colding \cite{Colding97, ChCo} ensures GH-convergence to a flat torus. This was extended by \cite{MMP22} to sequences of $\RCD$ spaces whose first Betti number coincides with the analytic dimension of the space. In \cite{CMT23}, the lower Ricci bound  was weakened to an integral Kato bound.

When considering scalar curvature, the situation is more involved. By the work of Schoen-Yau \cite{SchoenYau1979,SchoenYau1979-2} and Gromov-Lawson \cite{GromovLawson1980}, it is known that any metric on the torus with non-negative scalar curvature must be flat. The question of stability for this result was raised by Gromov. Namely, in page 61 in \cite{Gromov2014}, it is stated that ``there is a particular ``Sobolev type weak metric'' in the space of $n$-manifolds, such that, for example'' a sequence of tori with almost non-negative scalar curvature, combined with appropriate compactness conditions, should converge to a flat torus with respect to this metric. Several notions of convergence have been considered to address this question, starting from $C^0$ convergence of the Riemannian metric (see \cite{Gromov2014, Bamler2016, Burkhardt2019}). Sormani reformulated Gromov's conjecture \cite{Sormani_survey} by proposing a volume preserving condition and a uniform minA lower bound for $n=3$, where minA is the infimum of the 
areas of all closed minimal surfaces within the manifold. In doing so, she used the notion of intrinsic flat distance, which is a distance introduced by Sormani--Wenger \cite{SormaniWenger2011} between integral current spaces.  This led to studying the stability with respect to intrinsic flat distance of several classes of tori, as it was done, for example, in the following works: 
\begin{itemize}
\item 3-tori with a warped product metric \cite{AHPPS2019} (in this case, measured Gromov-Hausdorff 
and $C^0$ convergence is also obtained);
\item in dimension $n \geq 3$, graphical tori \cite{CPKP_graphTori} and conformal tori \cite{Allen2021,ChuLee}. 
\end{itemize}
For $n \geq 2$, Lee, Naber and Neumayer  \cite[Theorem 1.19]{LeeNaberNeu} obtained stability of tori with respect to the so-called $d_p$-convergence 
assuming that the entropy is small, and for $n\geq 4$ they constructed sequences of tori with small entropy that collapse in both Gromov-Hausdorff and intrinsic flat sense. More recently, Kazaras and Xu exhibited sequences of $3$-dimensional tori with a minA uniform lower bound that converge in the  $W^{1,p}$-sense, $1\leq  p < 2$, to a flat torus but do not converge in the intrinsic flat sense to a torus \cite{KX23}.

In this work, we establish new Gromov-Hausdorff stability results for tori by assuming a Ricci lower bound and an integral bound for the negative part of the scalar curvature. In particular, in conjunction with a recent result by Stern \cite{stern}, we obtain a new Gromov-Hausdorff stability result for $3$-tori. 
Observe that in the non-collapsing case, in which we mainly focus, Gromov-Hausdorff convergence also implies measured Gromov-Hausdorff convergence.

\subsection{Results}
Let $M$ be a Riemannian manifold and $\Phi$ be a harmonic map of the form
\begin{equation*}
\Phi=(\phi_1,\ldots, \phi_n):M\to \mathbb{S}^1(r_1)\times \cdots \times \mathbb{S}^1(r_n),
\end{equation*}
where $\mathbb{S}^1(r)=\{x \in \mathbb{R}^2\, : \, |x|_{\mathbb{R}^2}=r\}$.
Note that $\Phi$ determines $n$ harmonic $1$-forms $\{d\phi_i\}_{i=1}^n$ on $M$ because of the (isometrically) identification $\mathbb{S}^1(r)\cong \mathbb{R}/(2\pi r\mathbb{Z})$.

One of the main purposes of this paper is to consider $\Phi$ as a nonlinear analogue of a splitting map 
to a Euclidean space. Splitting maps have been extensively studied and applied in both 
Ricci limits and RCD theories. Notably, it is known that the existence of a splitting map is equivalent to the closeness to a Euclidean space.

In this paper, we define the notion of 
$(\delta; C, \tau)$-harmonic map as follows. Note that we can restrict our discussion to the case where each $\mathbb{S}^1(r_i)$ is isomorphic to $\mathbb{R}/\mathbb{Z}$ (equivalently, $\mathbb{S}^1(\frac{1}{2\pi})$) owing to the scaling factor applied to each circle. Namely we mainly work on $(\mathbb{R}/\mathbb{Z})^n=\mathbb{R}^n/\mathbb{Z}^n$ as the target space of harmonic maps (see also Remark \ref{new rem}).

\begin{definition}[$(\delta; C, \tau)$-harmonic map]
We say that a smooth map \mbox{$\Phi:M \to \mathbb{R}^n/\mathbb{Z}^n$}
is \textit{$(\delta; C, \tau)$-harmonic} if the following conditions are satisfied:
\begin{enumerate}
    \item $\Phi$ is a harmonic map;
    \item The {(averaged)} energy of $\Phi$ is bounded by $C$, 
    $$E(\Phi) := \frac{1}{2}{-\!\!\!\!\!\!\int_M}\langle \Phi^*g_{\mathbb{R}^n/\mathbb{Z}^n}, g_M\rangle \de\!\vol_M\le C;$$ 
    \item $\Phi$ is non-degenerate in the following sense,
    \begin{equation*}
D(\Phi):=\mathrm{det}\left( -\!\!\!\!\!\!\int_M\langle d\phi_i, d\phi_j\rangle \de\!\mathrm{vol}_M\right)_{ij}\ge \tau>0,
\end{equation*}
where $-\!\!\!\!\!\!\int_A \cdot =\frac{1}{\vol_M(A)}\int_A \cdot \de\!\mathrm{vol}_M$ when $0<\vol_M(A) <\infty$;
\item For $i=1,2,\ldots, n$, we have
\begin{equation*}
-\!\!\!\!\!\!\int_M|\nabla d\phi_i|^2\de\!\mathrm{vol}_M\le \delta.
\end{equation*}
\end{enumerate}
\end{definition}

\begin{remark}\label{new rem}
{ Given two flat tori $T, S$ of dimension $n$ it is also well-known that:
\begin{itemize}
    \item  there exists a canonical affine harmonic diffeomorphism $f:T\to S$;
    \item for a smooth map $\Phi$ from $M$ to  $T$, $\Phi$ is harmonic if and only if $f\circ \Phi$ is harmonic.
\end{itemize}}

Therefore all our results explained below can be also justified in the case when the target space is replaced by any flat torus. See also Remark \ref{bilip}.
\end{remark}
{
\begin{remark}\label{rem:delta}
    If $\ric_M\ge -\delta$ holds, then any harmonic map $\Phi:M \to \mathbb{R}^n/\mathbb{Z}^n$ satisfies
    $$\sum_{i=1}^n-\!\!\!\!\!\!\int_M|\nabla d\phi_i|^2\de \vol_M\le 2\delta E(\Phi).$$
    This is a direct consequence of Bochner's inequality (for functions) on $M$ and the flatness of $\mathbb{R}^n/\mathbb{Z}^n$. 
\end{remark}
}
We are now ready to state our main result concerning Gromov-Hausdorff stability of tori, 
which is directly linked to the harmonic map $\Phi$ described above. In the following we denote by $\Psi(\delta|C_1, \ldots , C_k)$ a quantity depending on $\delta, C_1, \ldots, C_k$ and tending to 0 as $\delta$ tends to zero.

\begin{theorem}[Characterization of almost flat tori via harmonic maps]\label{mthm0}
Let $n \in \mathbb{N}$, $ K \in \mathbb{R}, 0<\delta<1, D, L>0, {\tau>0}$ and let $M$ be a closed Riemannian manifold of dimension $n$ with $\mathrm{diam}_M\le D$ and $\mathrm{ric}_M\ge K$.
Then we have the following. 
\begin{enumerate}
\item If there exists a $(\delta;L, \tau)$-harmonic map \[\Phi=(\phi_i)_{i=1}^n:M \to \mathbb{R}^n/\mathbb{Z}^n\] 
then $M$ is $\Psi$-Gromov-Hausdorff close to a flat $n$-torus $T$, 
for some $\Psi=\Psi(\delta| n, K, D, L, \tau)$. 
 
Furthermore, for $\delta$ small enough 
depending only on $n, K, D, L$ and $\tau$, we have the following.
\begin{enumerate}
    \item $M$ is diffeomorphic to $T$;
    \item $\Phi$ is a covering map with $\mathrm{deg}(\Phi) \le C_0$; in particular $\Phi$ is surjective. More strongly, there exists an affine harmonic diffeomorphism $H:\mathbb{R}^n/\mathbb{Z}^n \to T$ such that $H$ is $C_1$-bi-Lipschitz (namely both $H, H^{-1}$ are $C_1$-Lipschitz) and that for all $x, y \in M$ with $\de_M(x, y) \le r$,
\begin{equation}\label{eq-holderlip}
C_2\de_M(x, y)^{1+\Psi} \le \de_{T}(H\circ \Phi(x), H \circ \Phi(y)) \le C_3\de_M(x, y),
\end{equation}
 where $C_i=C_i(K, n, D, L, \tau)>0$ and $r=r(K, n, D, L, \tau)>0$.
    \item if $\mathrm{deg}(\Phi)=1$, then $\Phi$ is a diffeomorphism, (\ref{eq-holderlip}) is satisfied for all $x, y \in M$, and $H \circ \Phi$ is a $\Psi$-Gromov-Hausdorff approximation.
\end{enumerate}

\item If $M$ is $\delta$-Gromov-Hausdorff close to a flat $n$-torus $T$ with $\mathrm{vol}_{T}(T) \ge v$ for some $v>0$, and
\begin{equation*}
-\!\!\!\!\!\!\int_M{r_M^-}\de\!\mathrm{vol}_M \le \delta,
\end{equation*}
{where $r_M(x)$ denotes the smallest eigenvalue of $\ric_M$ at x},
then there exists a  $(\Psi;C_3, C_4)$-harmonic diffeomorphism
\begin{equation*}
\Phi=(\phi_i)_{i=1}^n:M\to \mathbb{R}^n/\mathbb{Z}^n, 
\end{equation*}
where $C_i=C_i(K,n, D, v)>0$, $\Psi=\Psi(\delta|n, K, D, v)>0$ and { $r^-_M= \max\{-r_M, 0\}$}. 
\end{enumerate} 
\end{theorem}

\begin{remark}\label{rem-mthm0}
\begin{itemize}
\item[]
    \item In general $\Phi$ appearing in (1) above is not necessarily a diffeomorphism. For any $n \in \mathbb{Z}$ the covering map $$\phi_n:\mathbb{S}^1(1) \to \mathbb{S}^1(1),\quad (\cos \theta, \sin \theta) \mapsto (\cos n\theta, \sin n\theta)$$ is harmonic with  degree $n$ and $E(\phi_n)=\frac{\pi n^2}{2}$;
    \item {  In (2) above,  we conjecture that the integral $L^1$-bound assumption of the negative part of $r_M$ can be dropped. }
    \item From the proof of \cite[Theorem 1.2]{HondaPeng} and \cite[Theorem A.1.3]{ChCo}, if $M$ is $\epsilon$-Gromov-Hausdorff close to a flat torus $T$ with $\vol_{T}(T)\ge v>0$, then there exists a diffeomorphism $F:M \to T$ such that for all $x, y \in M$,
    \begin{equation}\label{pp}(1-\Psi)\de_M(x, y)^{1+\Psi}\le \de_T(F(x), F(y)) \le (1+\Psi)\de_M(x, y),
    \end{equation}
    and for any $1 \le p<\infty$
    \begin{equation}
    \|F^*g_{T}-g_M\|_{L^p}\le \Psi,
    \end{equation}
    where $\Psi=\Psi(\epsilon| K, n, D, v, p)$. Although (\ref{pp}) is sharper than (\ref{eq-holderlip}), the diffeomorphism obtained in Theorem \ref{mthm0} is a harmonic map.

    See also Proposition \ref{approx}.
\end{itemize} 
\end{remark}

In relation to the third bullet in Remark \ref{rem-mthm0},  
by a theorem of Anderson and Cheeger \cite{andersoncheeger}, 
if we additionally assume a lower bound for the injectivity radius, we get $C^{0,\alpha}$-closeness of the Riemannian metrics.

\begin{corollary}
Let $K\in \R, n \in \mathbb{N}$, $j, D, L, \epsilon, \tau>0$ and $\alpha\in (0,1)$. There exists $\delta>0$ such that the following holds.

If $M$ is a closed $n$-dimensional Riemannian manifold with $\ric_{M}\geq K$, $\mbox{inj}_{M}\geq j$, $\diam_M\leq D$,
and there exists a $(\delta; L, \tau)$-harmonic map $\Phi:M \to \mathbb{R}^n/\mathbb{Z}^n$, 
then $M$ is diffeomorphic to a flat torus, and there exists a flat smooth Riemannian metric $g_T$ on $M$ such that 
\[\left\|g_{M}- g_T\right\|_{C^{0, \alpha}}\leq \epsilon.\]
\end{corollary}

On the other hand, for a $3$-torus, by combining Stern's inequality from \cite{stern}, we obtain the following stability result with respect to Gromov-Hausdorff convergence. Note that in the theorem, we do not assume the existence of a $(\delta; L, \tau)$-harmonic map $\Phi$ into $\mathbb{R}^n/\mathbb{Z}^n$.

\begin{theorem}[Characterization of almost flat $3$-tori]\label{th:main}
Let $K \in \R$ and $D, v, \delta>0$. 
If $M$ is a Riemannian 3-torus with $\ric_{M}\geq K$, $\diam_M\leq D, \vol_M(M)\ge v$ and 
\[-\!\!\!\!\!\!\int_M R_M^-\de\!\vol_M\leq \delta,\]
then $M$ is $\Psi$-Gromov-Hausdorff close to a flat $3$-torus, where { $R^-_M= \max\{-R_M, 0\}$, $R_M$ denotes the scalar curvature of $M$} and $\Psi=\Psi(\delta |K, D, v)$.
\end{theorem}

Note that, by Gromov-Lawson and Schoen-Yau \cite{SchoenYau1979,SchoenYau1979-2, GromovLawson1980}, there is no metric on the torus such that $\ric\geq K$ for $K>0$. Furthermore in the case of $K=0$,  the only possibility is the flat metric. Thus, the result above provides an almost stability statement of this observation in terms of the Gromov-Hausdorff distance.

We point out that Allen, Bryden and Kazaras showed stability of 3-tori \cite[Theorem 1.11]{ABK} with respect to $C^{0, \gamma}$ convergence of the metrics under an $L^3$-integral Ricci curvature bound,  a two sided bound on the volume,  a uniform Neumann isoperimetric bound,
an $L^1$-integral bound on the negative part of the scalar curvature, and some topological restrictions. We prove an analogous result in our setting, see Theorem \ref{th:ABK}. We note that in dimension 3 the integral bound on Ricci is  equivalent to an integral bound of the full Riemannian curvature tensor. In contrast, the uniform lower bound on the Ricci tensor that we assume does not imply an integral bound for the Ricci nor the Riemann tensor, or any lower bound for the sectional curvature in general. We also note that the proofs of \cite{ABK} 
are based on the $C^{0, \gamma}$-compactness under an integral bound of Ricci curvature,  while our arguments use Gromov-Hausdorff compactness together with the analytic stability results for the Laplacian.

Observe that other works address stability questions by combining scalar and Ricci lower bound{s}: see for instance \cite[Theorem 1.1]{KazarasKhuriLee} for Gromov-Hausdorff stability of the positive mass theorem in dimension 3.

In both \cite{ABK} and \cite{KazarasKhuriLee} the restriction to dimension 3 is related to the use of a result by Stern \cite{stern} for smooth 3-manifolds. \\

{
Finally let us mention that our techniques also allow us to prove the following result in the case of almost nonnegative Ricci curvature.
\begin{theorem}\label{thm:almostnonricci}
    Let $M$ be an oriented closed Riemannian manifold of dimension $n$ with $\diam_M \le D$ and $\ric_M \ge -\delta$ 
    for some $D, \delta>0$. Then we have the following:
    \begin{enumerate}
        \item If there exists a harmonic map $\Phi:M \to \mathbb{R}^n/\mathbb{Z}^n$ such that $\haus^n(\Phi(M))\ge c>0$ and $E(\Phi) \le L$ hold for some $c, L>0$, then $M$ is $\Psi$-Gromov-Hausdorff close to a flat $n$-torus whose volume is bounded below by a positive constant depending only on $n, D, c$ and $L$, where $\Psi=\Psi(\delta|n, D, c, L)$. In particular $M$ is diffeomorphic to the torus if $\delta$ is small enough, depending only on $n, D, c$ and $L$.
        \item If there exists a smooth map $F:M \to \mathbb{R}^n/\mathbb{Z}^n$ whose degree does not vanish with $E(F) \le L$ for some $L>0$, then $M$ is $\Psi$-Gromov-Hausdorff close to a flat $n$-torus whose volume is bounded below by a positive constant depending only on $n, D$ and $L$, where $\Psi=\Psi(\delta|n, D, L)$. In particular, $M$ is diffeomorphic to the torus if $\delta$ is small enough, depending only on $n, D$ and $L$.
    \end{enumerate}
\end{theorem}}

\subsection{Strategy of proofs and remarks}

The main tools to prove Theorem \ref{mthm0} are:
\begin{enumerate}
    
    \item{(Theorem \ref{th:harmonicVF})} a strong convergent result in $L^2$ for harmonic vector fields;
    \item{(Theorem \ref{thm-torus})} a torus rigidity result in the spirit of Gigli and the fifth named author \cite{GR18};
    \item{(Theorem \ref{thm:new})} almost rigidity of Theorem \ref{thm-torus}.
    \end{enumerate}

The first point has been previously investigated in \cite{Honda15, Honda-PDE}. Nevertheless, we present a self-contained proof in Section 3, because the technique of this proof plays an important  role in establishing  Theorem \ref{thm:new}.

Regarding the second point, Gigli and Rigoni's original result states that if an $\RCD(0,N)$ space  with $N\in \N$ carries exactly $N$ harmonic vector fields (note that our working definition of harmonicity may be different from the usual one, namely our meaning is that the dual $1$-form is harmonic), it must be isomorphic to a flat $N$-torus, {namely, there exists a measure preserving isometry from the $\RCD(0, N)$ space to the flat torus equipped with its Riemannian distance and a constant multiple of the Riemannian measure.} Since their work focuses on the case in which the lower bound on the Ricci curvature is 0, while in our paper we consider spaces with Ricci curvature bounded from below by some $K \in \mathbb{R}$, we carefully modify certain aspects of Gigli and Rigoni's proof.

As for the almost rigidity of the previous result,  we need to establish $H^{1,2}_C$-stability for vector fields with respect to non-collapsed measured Gromov-Hausdorff convergence, which was not achieved in the setting of \cite{Honda17}. To get this, we first establish the smoothness and the flatness 
of the limit space, except for a closed subset having null $2$-capacity, where the last property comes from
a quantitative Minkowski estimate on the singular set, due to Cheeger, Jiang and Naber in \cite{CJN}. Then we follow techniques in \cite{Honda17} to get the desired $H_C^{1,2}$-stability.

Theorem \ref{thm:new}, when combined with 
the topological stability result by Cheeger-Colding, allows us to provide a characterization of almost flat tori in terms of harmonic $1$-forms. See Theorem \ref{partial answer} which has an independent interest. Based on this, we effectively realize Theorem \ref{mthm0}. It is important to emphasize that there are no non-collapsed assumptions in Theorem \ref{mthm0}, even though Theorem \ref{partial answer} is stated in a non-collapsed setting. The reason is simple: we can reduce the discussion, in the proof of Theorem \ref{partial answer}, to the case of functions via the Euler-Lagrange equation, allowing the application of stability results for the Laplacian acting on functions \cite{ambrosio_honda, ambrosio_honda_local, gms_convergence}.

To establish Theorem \ref{th:main}, we firstly  use the Ricci flow to regularize a given space, achieving bounded geometry as done in \cite{ST} by Simon-Topping. As a consequence, we construct a diffeomorphism $F$ from $M$ to $\mathbb{R}^3/\mathbb{Z}^3$ with a quantitative energy estimate.
Secondly, we use the harmonic map heat flow starting at $F$, established by Eells-Sampson in \cite{ES}, to construct a surjective harmonic map $G$ from $M$ to  $\mathbb{R}^3/\mathbb{Z}^3$ with the same quantitative energy estimate.
Then we use Stern's result in \cite{stern}, to prove the $\delta$-harmonicity for $G$. Here we use the topological assumption that $M$ is a torus.
Finally, to conclude from Theorem \ref{mthm0}, we need to find a quantitative positive lower bound on $D(G)$. This is justified by a contradiction argument based on the stability analysis on the Laplacian with the surjectivity (see Proposition \ref{prop-lower b}). Then we get Theorem \ref{th:main}.

It is crucial to note that the majority of results in the paper, including Theorems \ref{th:harmonicVF}, \ref{thm-torus}, and \ref{thm:new}, hold true in any dimension. The restriction to a dimension equal to 3 is only necessary in Theorem \ref{th:main} to be able to apply Stern's result.

In general we have the following conjecture. 

\begin{conjecture}\label{conjecture}
Let $K\in \R$, $n \in \mathbb{N}$, $v, D>0$ and $\epsilon >0$. There exists $\delta>0$ such that the following holds.
If $M$ is a  Riemannian $n$-torus with $\ric_{M}\geq K$, $\vol_{M}(M)\geq v$, $\diam_M\leq D$ and $\int_M R_M^-\de\!\vol_M\leq \delta$, then 
there exists a $n$-dimensional flat torus $T$ such that $M$ is $\epsilon$-Gromov-Hausdorff close to $T$
and that $M$ is diffeomorphic to $T$. 
\end{conjecture}

Our method, which combines the theories of harmonic maps, Ricci limits  and $\RCD$ spaces in conjunction with Stern's formula might also be applicable for other stability problems involving scalar curvature lower bounds in a Ricci-curvature-bounded-from-below background. We will investigate this in future work.\\

\textbf{Acknowledgments:} 
Part of this work was conducted at the Fields Institute for Research in Mathematical Sciences, where all the authors were partially supported to participate in the 2022 program on Nonsmooth Riemannian and Lorentzian Geometry. We would like to express our gratitude to the organizers of the program and to the Fields Institute for providing an excellent working environment. 
{
Moreover, we would like to thank the reviewer for their careful reading of the first version, which helped us improve the presentation of this manuscript.}
The first named author acknowledges supports of the Grant-in-Aid for Scientific Research (B) of 20H01799, the
Grant-in-Aid for Scientific Research (B) of 21H00977 and Grant-in-Aid for Transformative Research Areas (A) of 22H05105. This research was funded in part by the Austrian Science Fund (FWF) \href{https://doi.org/10.55776/ESP224}{10.55776/ESP224}. For open access purposes, the fifth author has applied a CC BY public copyright license to any author accepted manuscript version arising from this submission. Finally, we would like to acknowledge Man-Chun Lee for his comments about using Ricci flow to regularize manifolds, {and we wish to thank Andrea Mondino for pointing out an error in an earlier version of
Theorem \ref{partial answer}. }

\section{Preliminaries}

\subsection{Convergence of Sobolev functions on $\RCD$ spaces}\label{subsec:21}
In the sequel, we consider compact metric measure  spaces $(X_i, \de_{X_i}, \m_{X_i})= {\sf X}_i$ with $\m_{X_i}(X_i)=1$ that converge in measured Gromov-Hausdorff sense to a compact metric measure space $(X, \de_X, \m_X)= {\sf X}$. By definition, there exist a separable metric space $Y$ and distance preserving maps $\iota_i: X_i \rightarrow Y$, $\iota:X \rightarrow Y$ such that $(\iota_i)_{\#}\m_{X_i}$ weakly converge to $\iota_{\#}\m_X$ in $\mathcal M_{loc}(Y)$, where $\mathcal M_{loc}(Y)$ denotes the space of (signed) Borel measures on $Y$ that are finite on bounded sets. This space is equipped with the weak topology defined via duality with $C_{bs}(Y)$, the space of continuous functions with bounded support.  To simplify the notation, we will identify $\m_{X_i}$ and $\m_X$ with their pushforward under $\iota_i$ and $\iota$, respectively.

\begin{definition} 
Let $p\in (1, \infty)$. A sequence $f_i\in L^p(\m_{X_i})$ weakly converges in $L^p$ to $f\in L^p(\m_X)$ if $f_i \m_{X_i}$ weakly converges to $f \m_X$ in $\mathcal M_{loc}(Y)$ and 
\begin{align}\label{ineq:sup}\sup_{i\in \mathbb N}\left\| f_i\right\|_{L^p(\m_{X_i})}<\infty.
\end{align}

\begin{remark}
If $(f_i)$ weakly converges in $L^p$ to $f$, then $$\left\|f\right\|_{L^p(\m_X)}\leq \liminf_{i\rightarrow \infty} \left\|f_i\right\|_{L^p(\m_{X_i})}.$$
Given a sequence $f_i\in L^p(\m_{X_i})$ that satisfies \eqref{ineq:sup} there exists a subsequence of $(f_i)$ that weakly converges in $L^p$ to a function $f\in L^p( \m_X)$.
\end{remark}

A sequence $f_i\in L^p( \m_{X_i})$ strongly converges in $L^p$ to $f\in L^p( \m_X)$ if it weakly converges in $L^p$ and in addition satisfies
\begin{align*}
\limsup_{i\rightarrow \infty} \left\|f_i\right\|_{L^p( \m_{X_i})} \leq \left\| f\right\|_{L^p( \m_X)}.
\end{align*}
\end{definition}

\begin{remark}
All notions of strong and weak convergence have natural local analogues. These concepts derive immediately from the global ones by multiplying by characteristic functions. 
\end{remark}

From now on we assume that 
 ${\sf X}_i$ and ${\sf X}$  are compact   $\RCD(K,N)$ spaces for $K\in \R$ and $N\in (1, \infty)$. A  metric measure space satisfies the $\RCD(K,N)$ condition if it is infinitesimally Hilbertian \cite{giglistructure} and satisfies the curvature-dimension condition $\CD(K,N)$ in the sense of Lott-Sturm-Villani, which is a synthetic notion of Ricci curvature bounded from below by $K$ and dimension bounded from above by $N$ \cite{stugeo2, lottvillani, agsriemannian, erbarkuwadasturm, amsnonlinear}.
Two key results in the theory of $\RCD$ spaces are as follows: 
\begin{itemize}
\item The metric measure space associated to a Riemannian manifold satisfies the $\RCD(K,N)$ condition if and only if its Ricci curvature is bounded from below by $K$ and its dimension is bounded from above by $N$.
\item 
The $\RCD(K,N)$ condition is stable with respect to measured GH convergence.
\end{itemize}
In the following we collect properties of $\RCD$ spaces relevant for our purposes.  For more information we refer the reader to the survey paper \cite{villani_survey}. 
\smallskip

Let $\Ch_X$ be the Cheeger energy associated to ${\sf X}$ and $W^{1,2}({\sf X})$ the Sobolev space of $L^2(\m_X)$ functions with finite Cheeger energy \cite{cheegerlipschitz}. Given $U\subset X$ open, then  $f\in W^{1,2}_{loc}(U)$ if and only if for every Lipschitz function $\phi$ with compact support in $U$, $\phi \in \lip_c(U)$, it holds $\phi f\in W^{1,2}({\sf X})$. Every $f\in W^{1,2}_{loc}(U)$ admits a minimal weak upper gradient $|\nabla f|$ and $\Ch_X(f)= \int_X |\nabla f|^2 \,{\rm d}\!\m_X$ holds if  $f\in W^{1,2}({\sf X})$. 

Moreover if $f\in W^{1,2}_{loc}(U)$ is a locally Lipschitz function, then $|\nabla f|$ coincides for $\m_X$-a.e. with its local Lipschitz constant in $x$. 
The inner product $\langle \nabla f, \nabla g\rangle$ for $f,g\in W^{1,2}_{loc}(U)$ is defined by polarization of the minimal weak upper gradient $|\nabla \cdot|$. We also denote by $W^{1,2}(U)$ the set of all $f \in W^{1,2}_{loc}(U)$ with $f, |\nabla f| \in L^2(U)$. For more information (e.g. the Laplacian $\Delta_X$ and its domain $D(\Delta_X)$) we refer for instance to \cite{ags_heat}.

\begin{theorem}[\cite{gms_convergence}]\label{th:mosco} The Cheeger energies, $\Ch_i:=\Ch_{{\sf X}_i}$, Mosco converge to the Cheeger energy $\Ch:=\Ch_{\sf X}$, i.e. the following statements hold: 
\begin{enumerate}
\item For every $f_i\in L^2(\m_{X_i})$ weakly converging in $L^2$ to $f\in L^2(\m_X)$, one has 
\begin{align*}
\Ch(f)\leq \liminf_{i\rightarrow \infty} \Ch_i(f_i).
\end{align*}
Moreover, if $\sup_{i\in \mathbb N} \Ch_i(f_i)<\infty$, then $f_i$ strongly converges in $L^2$ to $f$. 
\item For every $f\in L^2(\m)$ there exists a sequence $f_i\in L^2(\m_i)$ that strongly converges in $L^2$ to $f$ with 
\begin{align*}
\Ch(f)= \lim_{i\rightarrow \infty} \Ch_i(f_i).
\end{align*}
Moreover, if $f\in \lip\big((X, \de)\big)\cap W^{1,2}({\sf X})$, then such $f_i$ can be chosen to satisfy $f_i\in \lip\big((X_i, \de_i)\big)\cap W^{1,2}({\sf X}_i)$ with $\sup_{i\in \mathbb N}\left\||\nabla f_i|\right\|_{L^\infty(\m_{X_i})}<\infty$. 
\end{enumerate}
\end{theorem}

The space of test functions is  
\[\Test F({\sf X})= \{f\in D(\Delta_X)\cap L^\infty(\m_X): |\nabla f|\in L^\infty(\m_X), \Delta_X f\in W^{1,2}({\sf X})\}.\]

\begin{definition}
A sequence $f_i\in W^{1,2}({\sf X}_i)$ weakly converges in $W^{1,2}$ to $f\in W^{1,2}({\sf X})$ if $f_i$ is weakly convergent  in $L^2$ (or equivalently strongly convergent in $L^2
$) to $f$ and $\sup_{i\in \mathbb N} \Ch_i(f_i)$ is finite. 

Strong convergence in $W^{1,2}$ is defined by requiring strong convergence  in $L^2$ of $f_i$ to $f$ and $\Ch(f)=\lim_{i\rightarrow \infty}\Ch(f_i)$. 
\end{definition}

\begin{lemma}[{\cite[Lemma 2.10]{ambrosio_honda}}] \label{lem:bumpfct}
Let $x_i\rightarrow x$.
For any $\phi \in \lip\big((X, \de)\big)$ with $\supp \phi\subset B_\delta(x)$ there exists  a sequence $\phi_i\in \lip\big((X_i, \de_i)\big)$ with $\supp \phi_i\subset B_\delta(x_i)$ such that $\sup_{i} \lip \phi_i<\infty$ and $\phi_i$ converges strongly in $W^{1,2}$ to $\phi$.
\end{lemma}

Let $R>0$ and let $x_i\rightarrow x$ in $Y$. We say that $f_i\in W^{1,2}(B_R(x_i))$ are weakly convergent in $W^{1,2}$ to $f\in W^{1,2}(B_R(x))$ if $f_i$ weakly (or equivalently strongly) converge in $L^2$ to $f$ with $$\sup_{i\in \mathbb N} \left\|f_i \right\|_{W^{1,2}(B_R(x_i))}<\infty.$$  Strong convergence in $W^{1,2}$ on $B_R(x)$ is defined by {additionally} requiring $$\lim_{i\rightarrow \infty} \left\| |\nabla f_i|\right\|_{L^2(B_R(x_i))} = \left\||\nabla f|\right\|_{L^2(B_R(x))}.$$

The following can be found in \cite[Theorem 4.2]{ambrosio_honda_local}.
\begin{theorem}\label{th:localcompactness}
Let $R>0$ and $f_i\in W^{1,2}(B_R(x_i))$ with $\sup_{i\in \mathbb N} \left\|f_i\right\|_{W^{1,2}(B_R(x_i))}<\infty$. Then there exists $f\in W^{1,2}(B_R(x))$ and a subsequence $f_{i_j}$ such that $f_{i_j}$ strongly converge in $L^2$ to $f$ on $B_R(x)$ and 
\begin{align*}
\liminf_{j\rightarrow \infty} \int_{B_R(x_{i_j})}|\nabla f_{i_j}|^2 \, {\rm d}\!\m_{X_{i_j}}\geq \int_{B_R(x)} |\nabla f|^2 \, {\rm d}\!\m_X.
\end{align*}
\end{theorem}

\begin{corollary}\label{cor:weak-conv}
Let $f_i,  (g_i)\in W^{1,2}(B_R(x_i))$ be strongly (respectively weakly) convergent in $W^{1,2}$ to $f, (g) \in W^{1,2}(B_R(x))$ on $B_R(x)$. Then 
\begin{align*}
\lim_{i\rightarrow \infty} \int_{B_R(x_i)} \langle \nabla f_i, \nabla g_i\rangle \,{\rm d}\!\m_{X_i} = \int_{B_R(x)} \langle \nabla f, \nabla g\rangle \,{\rm d}\!\m_X. 
\end{align*}
\end{corollary}

Let $U\subset X$ be open. A function $f\in W^{1,2}(U)$ is in the domain of the Laplacian,  i.e. $f\in D({\Delta}, U)$, if 
there exists a function ${\Delta}_U f\in L^2(U, \m_X|_U)$, sometimes denoted by $\Delta_X f$ for simplicity, such that 
\[\int_X \langle \nabla f, \nabla \phi \rangle \, {\rm d}\!\m_X = \int_X  \phi  {\Delta}_U f {\rm d} \!\m_X \qquad \forall \phi \in \lip_c(U).\]
We say that $f\in W^{1,2}(U)$ is harmonic if $f\in D({\Delta},U)$ and  ${\Delta}_{U} f=0$. This is consistent with the definition in \cite{ambrosio_honda_local}.

\begin{theorem}[{\cite[Theorem 4.4, Corollary 4.12]{ambrosio_honda_local}}]\label{th:localharmoniccompactness}
Let $f_i\in W^{1,2}(B_R(x_i))$ be harmonic functions on $B_R(x_i)$. If $(f_i)_i$ strongly converges in $L^2$ to $f$ on $B_R(x)$ with $\sup_{i\in \mathbb N}\left\|f_i\right\|_{W^{1,2}(B_R(x_i))}<\infty$, then $f$ is also harmonic on $B_R(x)$. 
Moreover, $f_i|_{B_r(x_i)}$ strongly converge in $W^{1,2}$ to $f|_{B_r(x)}$ on $B_r(x)$ for all $r\in (0,R)$. 
\end{theorem}

\begin{corollary}\label{cor:pairing}
Let $f_i\in W^{1,2}(B_R(x_i))$ be weakly converging  in $W^{1,2}$ to $f\in W^{1,2}(B_R(x))$ on $B_R(x)$. Let $r\in (0,R)$ and let $h_i\in W^{1,2}(B_r(x_i))$ strongly converge in $W^{1,2}$ to $h\in W^{1,2}(B_r(x))$ on $B_r(x)$ with $\sup_{i\in \mathbb N} \left\| |\nabla h_i|\right\|_{L^{\infty}(\m_i)}<\infty$. Then $\langle \nabla f_i, \nabla h_i\rangle$ 
in $B_r(x_i)$ weakly converges  in $L^2$ to $\langle \nabla f, \nabla h\rangle$ in $B_r(x)$.
\end{corollary}

\begin{proof} Weak convergence in $W^{1,2}$ of $(f_i)$ on $B_R(x_i)$ yields that 
\[\sup_{i\in \N} \left\|{f_i}\right\|_{W^{1,2}(B_R(x_i))}\leq C <\infty\]
for some $C\in (0, \infty)$.
Hence $\langle \nabla f_i, \nabla h_i\rangle$ is uniformly $L^2$-bounded on $B_r(x_i)$. Then  we can apply  \cite[Theorem 4.6]{ambrosio_honda_local}.
\end{proof}
The next theorem is obtained as a corollary of Theorem \ref{th:localcompactness} (see also  \cite[Theorem 4.4]{ambrosio_honda_local}).
\begin{theorem}\label{th:211} Assume each ${\sf X}_i$ is an $\RCD(K,N)$ space with $N\geq 2.$
Let $x_i\in X_i$ be any sequence of points and 
$f_i\in W^{1,2}(B_R(x_i))$ any sequence of harmonic functions on $B_R(x_i)$ such that  
 $\int_{B_R(x_i)} f_i \, {\rm d}\!\m_{X_i}=0$. Assume that $\sup_{i\in \mathbb N} \left\| |\nabla f_i|\right\|_{L^2(B_R(x_i))}<\infty$. Then  there exists a subsequence  $(i_j)_{j\in \mathbb N}$ such that $x_{i_j}\rightarrow x$
 and there exists a harmonic function $f\in W^{1,2}(B_{R}(x))$ such that $f_{i_j}|_{B_r(x_{i_j})}$ strongly converge in $W^{1,2}$ to $f$ on $B_r(x)$ for any $r\in (0,R)$.
\end{theorem}

\begin{proof} Since ${\sf X}_i$ GH converges to $\sf X$ we can extract a subsequence $(i_j)_j$ such that  $(x_{i_j})$ converges to a point $x\in X$. 

Recall the Sobolev inequality 
\[\left(
-\!\!\!\!\!\!\int_{B_R(x_i)} \left| f- 
-\!\!\!\!\!\!\int_{B_R(x_i)} \!\!f_i \de\!\m_{X_i}\right|^{2^*}\de\!\m_{X_i}\right)^{\frac{1}{2^*}} \!\!\!\leq
C\left(-\!\!\!\!\!\!\!\int_{B_R(x_i)} |\nabla f_i |^2 \de\!\m_{X_i}\right)^{\frac{1}{2}}
\] where $C=C(K,N,2^*, R)>0$, $2^*= \frac{2N}{N-2}<\infty$ for $N>2$, and $2^*$ is any number in $(2, \infty)$ for $N=2$. In our setting the inequality follows from the local Poincar\'e inequality \cite[Theorem 1]{rajala2},  \cite[Theorem 5.1]{koskela} and the Bishop-Gromov estimate on $\RCD$ spaces.

Together with $\int_{B_R(x_i)}f_i \de\!\m_{X_i}=0$ and $\sup_i\left\| |\nabla f_i|\right\|_{L^2(\m_{X_i})}<\infty$ we have that the assumptions of Theorem \ref{th:localcompactness} hold. Thus, after extracting another subsequence, $f_{i_j}$ converges $L^2$ strongly to $f\in W^{1,2}(B_R(x))$ on $B_R(x)$. Hence, by Theorem \ref{th:localharmoniccompactness} $W^{1,2}$-strong convergence in $B_r(x)$ follows. The harmonicity of $f$ is a direct consequence of \cite[Theorem 4.4]{ambrosio_honda_local} (or Lemma \ref{lem:div}).
\end{proof}

\subsection{Convergence of vector fields}
We assume familiarity with the $L^2$-tangent and cotangent module $L^2(T \sf X)$ and $L^2(T^\ast\sf X)$ associated to a metric measure space $\sf X$. For details we refer to \cite{giglinonsmooth}. 
If $\sf X$ comes from a smooth Riemannian 
manifold endowed with a smooth reference
measure, then the tangent module can be canonically identified with the space
of $L^2$-vector fields. 

If $\sf X$ is infinitesimal Hilbertian, then $L^2(T^*\sf X)$ and $L^2(T\sf X)$ are Hilbert modules and $L^2(T^*{\sf X})\simeq L^2(T{\sf X})$. Let $\sharp: L^2(T{\sf X})\rightarrow L^2(T^* {\sf X})$ be the {\it musical isomorphism} that identifies the tangent with the cotangent module. 

\smallskip
The following definitions and results are taken from \cite{ast}.
\begin{definition}
A sequence of vector fields $V_i\in L^2(T {\sf X}_i)$ weakly converges to $V\in L^2(T{\sf X})$ in $L^2$ if  $\sup_{i\in \mathbb N} \int |V_i|^2 \, {\rm d}\!\m_{X_i}<\infty$ and $\langle V_i, \nabla h_i\rangle \m_{X_i}$ weakly converges to $\langle V, \nabla h\rangle \m_X$ whenever $h_i\in W^{1,2}({\sf X}_i)$ strongly converges in $W^{1,2}$ to $h\in W^{1,2}({\sf X})$ and $\sup_{i\in \mathbb N} \left\| |\nabla h_i|\right\|_{L^\infty(\m_{X_i})}<\infty$. 

We say that  $V_i\in L^2(T{\sf X}_i)$ strongly converges in $L^2$ to $V\in L^2(T{\sf X})$ if, in addition, $\lim_{i\rightarrow \infty}\int_{X_i}|V_i|^2\,{\rm d}\!\m_{X_i}= \int_X |V|^2 \, {\rm d}\!\m_X$. 
\end{definition}

\begin{proposition}\label{prop-L2weakV} 
Let ${\sf X}_i$ be $\RCD(K,N)$ spaces that converge in measured Gromov-Hausdorff sense to an $\RCD(K,N)$ space $\sf X$.
Let $V_i\in L^2(T{\sf X}_i)$ be vector fields such that $\sup_{i\in \mathbb N} \int_{X_i} |V_i|^2 {\rm d}\!\m_{X_i}< \infty.$ Then, there exists a subsequence $V_{i_j}$ of $V_i$ and $V\in L^2(T{\sf X})$ such that $V_{i_j}$ weakly converge  in $L^2$ to $V$. 
\end{proposition}

\subsection{Harmonic vector fields}\label{harmonic vec} 

Let $\sf X$ be an $\RCD(K,N)$ space. 
The space of Sobolev vector fields, denoted as $W_{C}^{1,2}(T\sf X)$, is defined as the space of vector fields $V\in L^2(T \sf X)$ such that its covariant derivative in the sense of \cite{giglinonsmooth}, $\nabla V\in L^{2}(T^{\otimes 2}\sf X)$, exists. Moreover on $W_{C}^{1,2}(T\sf X)$, we define the energy $\mathcal E(V)=\int_X |\nabla V|^2_{HS} {\rm d}\!\m_X$, where $|\nabla V|_{HS}$ is the Hilbert-Schmidt norm of $\nabla V$.

The space of vector fields with finite energy $\mathcal E(V)$ is denoted as $W_{C}^{1,2}(T{\sf X})$. The $W_{C}^{1,2}(T{\sf X})$-closure of the linear span of vector fields of the form $g\nabla f$ with $g,f\in \Test F({\sf X})$ is denoted by $H^{1,2}_C(T{\sf X})$. 

We also denote by $D(\mathrm{div})$ the space of $L^2$-vector fields $V \in L^2(T{\sf X})$ satisfying that there exists a unique $f \in L^2(\m_X)$, denoted by $\mathrm{div} (V)$, such that
$$\int_X\langle \nabla \phi, V\rangle \de\!\m_X=-\int_X\phi f\de\!\m_X$$
holds for any $\phi \in W^{1,2}({\sf X})$.

Similarly, the space $W_{ d}^{1,2}(T^\ast {\sf X})$ is the space of $1$-forms $\omega \in L^2(T{\sf X})$ such that the exterior differential in the sense of \cite{giglinonsmooth} , ${ d} \omega \in L^2(\Lambda^2 T^\ast {\sf X})$,  exists. In the same manner,  we define $H_{ d}^{1,2}(T^\ast {\sf X}) \subset W_{ d}^{1,2}(T^\ast {\sf X})$ as the $W^{1,2}_{ d}$-closure of the set of all the linear combinations of $1$-forms of the form $f { d} g$ for $f, g \in \Test F({\sf X})$. It follows from \cite[(3.5.2)]{giglinonsmooth} that 
$$d\eta(V_0, V_1)=(\nabla_{V_0}\eta)(V_1)-(\nabla_{V_1}\eta)(V_0)$$
holds for all $V_i \in L^2(T{\sf X}), \eta=fdg$ for $f, g \in \Test F({\sf X})$. In particular
$H^{1,2}_C(T^*{\sf X})\subset H^{1,2}_{d}(T^*{\sf X})$ holds with $|d \omega| \le C|\nabla \omega|$ for $\m_X$-a.e., thus
\begin{equation}\label{dc ineq}
\|d \omega\|_{L^{2}}\le C\|\nabla \omega\|_{L^{2}},
\end{equation} where $C>0$ is a universal positive constant and $H^{1,2}_C(T^*{\sf X})$ is also a Sobolev space for $1$-forms, which is isomorphic to $H^{1,2}_C(T{\sf X})$ via the musical isomorphism $T{\sf X}\cong T^*{\sf X}$ (similar notations, including $W^{1,2}_d(T{\sf X})$, will be used immediately below).

The Hodge Laplacian $\Delta_H$ for vector fields in $ W^{1,2}_{C}(T{\sf X})$ is defined analogously to the Laplace operator for functions,  replacing the Cheeger energy with the Hodge energy \cite{giglinonsmooth}. Namely, letting $W^{1,2}_H(T{\sf X})=W^{1,2}_{\de}(T{\sf X}) \cap D(\mathrm{div})$ with $\|V\|_{W^{1,2}_H}^2=\|V\|_{W^{1,2}_d}^2+\|\mathrm{div}(V)\|_{L^2}^2$, we denote by $H^{1,2}_H(T{\sf X})$ the $W^{1,2}_H$-closure of the space of test vector fields, and then
a vector field $V$ is called harmonic if $V \in H^{1,2}_H(T{\sf X})$ holds with $\mathrm{div}(V)=0$ and $dV^{\flat}=0$.

Note that our working definition of a harmonic vector field is different from a traditional one,  equivalently, we mainly work on harmonic $1$-forms (namely,
applying the $\flat$ operator to a harmonic vector field $V$ we obtain a harmonic $1$-form $V^\flat$ that satisfies ${\rm d}^{\star}V^\flat=0$ and ${\rm d} V^{\flat}=0$.

In the context of a smooth Riemannian manifold ${\rm d}$ is the exterior derivative on forms and ${\rm d}^\star$ the corresponding dual operator with respect to the $L^2$-inner product).

It is worth mentioning that Bochner's inequality shows $H^{1,2}_H(T{\sf X}) \subset H^{1,2}_C(T{\sf X})$ with $\|V\|_{H^{1,2}_C}\le C(K)\|V\|_{H^{1,2}_H}$ \cite[Corollary 3.6.4]{giglinonsmooth}. Let us recall the following.
\begin{lemma}\label{h=c}
    If $\sf X$ is non-collapsed, namely $\m_X=\haus^N$, then $H^{1,2}_H(T^*{\sf X})=H_{C}^{1, 2}(T^*\sf X)$ with $|dV^{\flat}|+|\mathrm{div}(V)|\le C(N) |\nabla V|$ for $\m_X$-a.e. and {$\|V\|_{H^{1,2}_H}\le C(N)\|V\|_{H^{1,2}_C}$.}
\end{lemma}
\begin{proof}
    It follows from \cite[Proposition 3.2]{Han19} that for any test vector field $V$, we have $\mathrm{div}(V)=\mathrm{tr}(\nabla V)$. From this with (\ref{dc ineq}), we have the conclusion. 
\end{proof}

Finally let us mention that local analogues, e.g. $W_{ d,loc}^{1,2}$, are also well-defined as in the case of functions.

The following is a direct consequence of Theorem 4.1 in \cite{Honda15}, {where although the paper deals with Ricci limit spaces, the same proof works in this setting}. For the convenience of the reader, we provide a proof.

\begin{lemma}\label{lem:div}
Let ${\sf X}_i$ and ${\sf X}$ be as in the previous subsection and, let $(V_i)$ be a sequence of vector fields on ${X_i}$ with $V_i \in D(\mathrm{div})$ and $\sup_i\|\mathrm{div}(V_i)\|_{L^2}<\infty$ that converges weakly in $L^2$ to a vector field $V {\in L^2(T \sf X)}$ on ${X}$. 
Then $V \in D (\mathrm{div})$ and $\mathrm{div} (V_i)$ $L^2$-weakly converge to $\Div (V)$. 
\end{lemma}
\begin{proof}

By Theorem \ref{th:mosco} for any  Lipschitz function $h \in W^{1,2}({\sf X})$  we can pick a sequence of Lipschitz functions $(h_i)_i$ in $W^{1,2}({\sf X}_i)$ with $\sup_i \left\| |\nabla h_i |\right\|_{L^\infty(\m_{X_i})}<\infty$ that strongly converges in $W^{1,2}$ to $h$. By definition of weak convergence of $(V_i)_i$ it follows that $\langle V_i, \nabla h_i\rangle \m_{X_i}$ converges weakly to $\langle V, \nabla h\rangle \m_X$. In particular 
\begin{equation}
\int_{X_i} \langle V_i, \nabla h_i \rangle \de\!\m_{X_i} \to \int_X \langle V, \nabla h \rangle \de\!\m_X.
\end{equation}
On the other hand, after passing to a subsequence, we can find the $L^2$-weak limit $f$ of $\mathrm{div}(V_i)$. Thus we have
\begin{equation}
\int_{X_i}\mathrm{div}(V_i) h_i \de\!\m_{X_i} \to \int_Xf h\de\!\m_X.
\end{equation}
Therefore
\begin{equation}
\int_X\langle V, \nabla h\rangle \de\!\m_X=-\int_X fh\de\!\m_X,
\end{equation}
which shows $V \in D(\mathrm{div})$ with $\mathrm{div}(V)=f$. This completes the proof.
\end{proof}

\begin{theorem}\label{thm-Honda}
If the vector fields  $V_i\in W^{1,2}_C(T{\sf X}_i)$ strongly converge  to $V {\in L^2(T \sf X)}$ in $L^2$ and $\sup_i \left\| \nabla V_i\right\|_{L^2}< \infty$, then $V\in W^{1,2}_C(T{\sf X})$ and $\nabla V_i$  weakly converges to $\nabla V$ in $L^2$. Moreover 
$$\liminf_{i\rightarrow \infty} \int_{X_i} |\nabla V_i|^2 \de\!\m_{X_i}\geq \int_{X}|\nabla V|^2\de\!\m_X.$$
\end{theorem}
\begin{proof}
If the spaces  ${\sf X}_i$ emerge as non-collapsed Gromov Hausdorff limits of Riemannian manifolds the theorem is proven by Honda in \cite[Theorem 6.7]{Honda-PDE}. If ${\sf X}_i$ are $\RCD$ spaces, the theorem still holds by adapting the proof, given by Ambrosio and Honda in \cite[Theorem 10.3]{ambrosio_honda}, for weak convergence of the Hessian operators of a sequence of functions $f_i$ that converges strongly in $W^{1,2}$.  For this one observes that $\nabla V$ for a vector field $V$ in $W_C^{1,2}$ is defined by the same duality formula as the Hessian operator of a function $f$ in $W^{2,2}$ (Definition 3.3.1 and Definition 3.4.1 in \cite{giglinonsmooth}).
\end{proof}

Similarly we have the following (see also \cite{Honda-PDE}).
\begin{theorem}\label{th:stablity d}
If the $1$-forms  $\omega_i\in W^{1,2}_d(T^*{\sf X}_i)$ strongly converge to $\omega {\in L^2(T^* \sf X)}$ in $L^2$ and $\sup_i \left\| d\omega_i\right\|_{L^2}< \infty$, then $\omega\in W^{1,2}_d(T^*{\sf X})$ and $d \omega_i$  weakly converges to $d\omega$ in $L^2$. Moreover 
$$\liminf_{i\rightarrow \infty} \int_{X_i} |d\omega_i|^2 \de\!\m_{X_i}\geq \int_{X}|d\omega|^2\de\!\m_X.$$
\end{theorem}

\subsection{Monotone vector fields}

We introduce the notion of monotone vector fields which will be used to study the
behavior of the $W_2$ distance under the Regular Lagrangian Flow.

\begin{definition}
Consider $U\subset X$ open, {a real number $k \in \mathbb{R}$},  and a vector field $V\in W^{1,2}_{C, loc}(T{\sf X})$. We say $V$ is infinitesimally $k$-monotone  in $U$ if for any $ W\in H^{1,2}_C(T{\sf X})$, $W \equiv 0$ on $U^c$, it holds
\[\nabla_{sym} V(W\otimes W) \geq k|W|^2  \ \m\mbox{-a.e. in $U$}, \]
where we use the symbol $\nabla_{sym}$ to denote the symmetric  derivative, i.e. \[2\nabla_{sym} V(W\otimes Z) = \langle\nabla_W V, Z\rangle + \langle \nabla_Z V, W\rangle \quad \text{for all } \, W, Z \in H^{1,2}_C(T \sf X).\]
\end{definition}

\noindent Let us denote by $\mathcal P_2(X, \m)$  the set of measures in $\mathcal P(X, \m)$ with finite second moment.

\begin{definition}
    Consider $U\subset X$ open and a vector field $V\in L^2_{loc}(T{\sf X})$. We say that $V$ is $k$-monotone in $U$ if 
    \[ \int_X {\rm d} \phi (V) {\rm d} \mu^0 + \int_X {\rm d} \phi^c (V) {\rm d}\mu^1 \geq k W_2(\mu^0, \mu^1)^2\]
    for any pair $\mu^0, \mu^1\in \mathcal P_2(X, \m)$ with bounded densities and bounded supports in $U$ and any couple of Kantorovich potentials $(\phi, \phi^c)$  relative to $(\mu^0, \mu^1)$. 
    \end{definition}

\begin{remark}\label{rem-monotoneiff}
Let $\sf X$ be an $\RCD(K, N)$ space and let $V\in W^{1,2}_{C, loc} (T{\sf X})$. Then $V$ is $k$-monotone  in $X$ if and only if $V$ is infinitesimally $k$-monotone in $X$. 
For later purposes we check the ``if'' direction in this statement.

Let $\mu_0, \mu_1\in\mathcal P(X, \m_X)$ be concentrated in $B_{R}(x_0)$. Then the  $W_2$-geodesic $(\mu_t)_{t\in [0,1]}$ between them is concentrated in $B_{2R}(x_0)$. Let $\Pi$ be an optimal geodesic  plan such that $(e_t)_{\#}\Pi= \mu_t$ and $\phi_t$ be a function such that $-(s-t) \phi_t$ is a Kantorovich potential  relative to $(\mu_t, \mu_s)$. By the second variation formula { \cite{GT_sv_short, GT_sv_long}} we have that $t\in [0,1]\mapsto\langle V, \nabla \phi_t\rangle \circ e_t\in L^2(\Pi)$ is in $C^1([0,1], L^2(\Pi))$ and 
\begin{align*}
\frac{\rm d}{{\rm d}t} \langle V, \nabla \phi_t\rangle \circ e_t= \nabla_{sym} V(\nabla \phi_t \otimes \nabla\phi_t)\circ e_t .
\end{align*}
Integrating the previous equation with respect to $\Pi$ and with respect to $t$ from $0$ to $1$ yields
\begin{align*}
\int_X \langle V, \nabla \phi^c\rangle {\rm d}\mu_1+ \int_X\langle V, \nabla \phi \rangle {\rm d}\mu_0&=
\int_0^1 \int_X 
\frac{\rm d}{{\rm d}t} \langle V, \nabla \phi_t\rangle {\rm d}\mu_t {\rm d}t\\
&= \int_0^1 \int_X \nabla V(\nabla \phi_t \otimes \nabla\phi_t) {\rm d}\mu_t  {\rm d}t,
\end{align*}
where $\phi=-\phi_0$ and $\phi^c= \phi_1$. Since $V$ is infinitesimally $k$-monotone, the RHS is bounded from below by 
\[k \int_0^1\int_X |\nabla \phi_t|^2 {\rm d}\mu_t  {\rm d}t = k W_2(\mu_0, \mu_1)^2. \]
Hence $V$ is $k$-monotone.
\smallskip
\end{remark}

\begin{corollary}\label{cor:monV}
Consider $V\in W^{1,2}_{C,loc}(T{\sf X})$ such that $\nabla_{sym} V= 0$ $\m_X$-a.e. in $X$.  Then $V$ and $-V$ are $0$-monotone in $X$. 
\end{corollary}

\subsection{Regular Lagrangian Flows}

In this section we recall the basic notions needed to prove the rigidity of tori in Section \ref{sec:torusrigidity}. For  details on the following definition and Theorem \ref{thm:RLF} we refer to \cite{amtr, amtrnotes}.

\begin{definition}
Let $\sf X$ be an $\RCD(K,N)$  space  and  $V\in L^2([0, T], L^2_{{loc}}(T{\sf X}))$ be a time-dependent vector field. We say that the measurable map
\[F \colon [0, T]\times {\sf X}\to {\sf X}\]
is a Regular Lagrangian flow (RLF) associated to $V$ if 
\begin{enumerate}
\item  $\left(F^s\right)_{\sharp}\m\leq C\m$ $\forall s\in [0,T]$ for some $C=C(T)>0$;
\item  for $\m$-a.e. $x\in X$ the curve $s\in [0,T]\mapsto F^s(x)$ is continuous and $F^0(x)=x$;
\item for each $f\in W^{1,2}({\sf X})$ and for $\m$-a.e. $x\in X$, the function $s\mapsto f(F^s(x))$ belongs to $W^{1,1}([0, T])$ and 
\[\frac{\rm d}{{\rm d} s} f(F^s(x))={\rm d}\!f (V_s)(F^s(x)) \mbox{ for } \mathcal L^1\mbox{-a.e. } s \in [0, T].\]
\end{enumerate}
\end{definition}

\begin{theorem}\label{thm:RLF} 
Let $V: [0,T] \rightarrow L^2_{loc}(T{\sf X})$ be a Borel vector field such that $V_t\in D(\Div)$ for every $t\in [0,T]$. Assume that $\left\||V_t|\right\|_{L^2(\m_X)} \in L^1([0,T]), \left\|\Div(V_t)^-\right\|_{L^\infty}\in L^\infty([0,T])$ and $\left\|\nabla_{sym} V_t\right\|_{L^{2}(T^{\otimes 2}X)}\in L^1([0,T])$. Then there exists $F \colon [0, T] \times {\sf X}\to {\sf X}$, RLF associated to $V$, such that
\begin{itemize}
    \item[$(i)$] for any initial condition $\mu_0 = f\!\m_X$ with $f \in L^1(\m_X) \cap L^\infty(\m_X)$, $\mu_t := (F_t)_\sharp \mu_0$ is a solution to the following continuity equation
    \begin{equation}\label{eq:CE}
    \dfrac{\rm d}{{\rm d} t} \int_X g \, \de\!\mu_t = \int_X \langle \nabla g, V_t \rangle \, \de\!\mu_t, \quad \mathcal L^1\mbox{-a.e. } t \in (0, T)
    \end{equation}
    with $\lim_{t \to 0} \int g \, \de\! \mu_t = \int g \,\de\! \mu_0$, for any $g \in \lip(X, \de) \cap L^\infty(\m_X)$;
    \item[$(ii)$] for any initial condition $\mu_0 = f\!\m_X$ with $f \in L^1(\m_X) \cap L^\infty(\m_X)$, $\frac{\de\!\mu_t}{\de\!\m_X} \in L^1(\m_X) \cap L^\infty(\m_X)$ for any $t \in (0, T)$, while if $\mu_0$ is a probability measure with $f \in L^2(\m_X)$, then for any $t \in (0, T)$ it holds
    \[
    \bigg|\bigg| \dfrac{\de\!\mu_t}{\de\!\m_X}   \bigg|\bigg|_{L^2(\m_X)} \le e^{D t} || f ||_{L^2(\m_X)}
    \]
    where the constant $D$ depends only on $|| (\Div V_t)^- ||_{L^\infty(\m_X)}$;
    \item[$(iii)$] for $\m_X$-a.e. $x$ the curve $s\in [0,T]\mapsto F^s(x)$ is absolutely continuous and its metric speed is given by \[|\dot{(F^s(x))}|= |V_t|\circ F^t(x) \mbox{ for $\mathcal L^1$-a.e. } t\in [0,T].\]
    \item[$(iv)$] $F_t$ is unique, where uniqueness is to be intended in the following sense: if $\tilde F_t$ is another map satisfying the above properties, then $(\tilde F_t)_\sharp \mu = (F_t)_\sharp \mu$ for any probability measure $\mu \in \mathcal P(X, \m_X)$ with bounded density, where $\mathcal P(X, \m_X)$ is the set of probability measures in $X$ that are absolutely continuous with respect to $\m_X$.
    \end{itemize}
\end{theorem}

\begin{remark}
    We point out that as a direct consequence of the last property in Theorem \ref{thm:RLF}, we get that $(F_t)_{t \in [0, T]}$ is a semigroup, namely
    \[
    F_{t + s}(x) = F_t \circ F_s(x) \, \m_X\mbox{-a.e.\,} x \in X, \, \forall s, t, t+s \in [0, T].
    \]
\end{remark}

\subsection{Sobolev and harmonic maps into circles}\label{circleSob}

In this subsection, we provide a brief introduction to the Sobolev space $W^{1,2}(X; \mathbb{S}^1(r))$, of maps from a compact $\RCD(K, N)$ space, denoted as $\sf X$, to the circle $\mathbb{S}^1(r)$.

We say that a map $f:X \to \mathbb{S}^1(r)$ is a Sobolev map if each component $f_i$ is in $W^{1,2}({\sf X})$, where $f=(f_1, f_2)$ is considered as a map into $\mathbb{R}^2$ via the canonical inclusion $\mathbb{S}^1(r) \hookrightarrow \mathbb{R}^2$. 
Note that by the locality of the minimal relaxed slope, $|\nabla f| \in L^2(\m_X)$  makes sense through the identification $\mathbb{S}^1(r)\cong \mathbb{R}/(2 \pi r\mathbb{Z})$. Consequently, we can state similar definitions and results for $W^{1,2}(X; \mathbb S^1(r))$ as for $W^{1,2}(X)$ in Section 3.1. 
In particular, we can define the differential $1$-form $df \in L^2(T^*{\sf X})$. We {then} define $f$ to be harmonic if it is a critical point of the {(averaged)} energy functional,
\begin{equation}
E(f):=\frac{1}{2}{-\!\!\!\!\!\!\int_X}e(f)\de\!\m_X,\quad e(f):=|df_1|^2+|df_2|^2 (=|\nabla f|^2).
\end{equation}
Furthermore, if $f$ is Lipschitz, it can be deduced, for instance, from \cite[Proposition 6.2]{HS}, that $f$ is harmonic if and only if $f_i \in D(\Delta_X)$ holds, and the Euler-Lagrange equation:
\begin{equation}\label{EL}
r^2\Delta_Xf_i+|\nabla f|^2f_i=0
\end{equation}
is satisfied for all $i=1,2$. This observation extends to cases where the target space is $\mathbb{S}^1(r_1)\times \cdots \times \mathbb{S}^1(r_k)$. It is worth mentioning that the local Lipschitz continuity of a harmonic map from an open subset into a (locally) $\mathrm{CAT}(0)$ space has been established in \cite{Gigli-Lipschitz, MS}.

\section{Strong convergence of harmonic vector fields} 
The aim of this section is to improve 
Proposition \ref{prop-L2weakV}  to strong convergence. For this purpose, we impose smoothness of the $\RCD$ spaces and harmonicity of vector fields. The proof of this result is an adaptation of \cite[Theorem 6.9]{Honda-PDE}. We provide a self-contained proof for the convenience of the reader, as the techniques will be used in the proofs of the main results.

\begin{theorem}\label{th:harmonicVF} 
Let $\{M_i\}_{i \in \mathbb N}$ be a sequence of compact $N$-dimensional Riemannian manifolds with $\ric_{M_i}\geq K$, $\diam_{M_i}\leq D$ and $\vol_{M_i}(M_i)\geq v$ that converges in Gromov-Hausdorff sense to a non-collapsed $\RCD(K,N)$ space $\sf X$.
Let $V_i\in L^2(TM_i)$ be harmonic vector fields with 
\[\sup_{i\in \mathbb N} \int_{M_i} |V_i|^2 {\de}\!\vol_{M_i}< \infty.\]
Then there exists a subsequence $V_{i_j}$ and a vector field $V\in L^2(T{\sf X})$ such that $V_{i_j}$ strongly converges in $L^2$ to $V$.
\end{theorem}

We recall the definition of regular points.  
Let $N\in \mathbb N$. Given a complete metric measure space ${\sf X} = (X, \de, \m)$ and $r, \epsilon>0$, following \cite{ChCo, kapovitch_mondino}, we define
{the set of $(N,\epsilon,r)$-regular points of $X$ by}
\[(\mathcal R_N)_{\epsilon, r} = \Big\{ x\in X: \exists t>r \mbox{ such that } \de_{GH}\big(B^X_s(x) , B_s^{\mathbb R^N}(0)\big)\leq \epsilon s\  \forall s\in (0, t)\Big\}.\]
The $(\epsilon, N)$-regular set of $X$ is
\[(\mathcal R_N)_\epsilon = \bigcup_{r>0} (\mathcal R_N)_{\epsilon, r}\]
and the set of $N$-regular points of $\sf X$ is 
\[\mathcal R_N= \bigcap_{\epsilon>0} (\mathcal R_N)_{\epsilon}.\]

Finally recall that an $\RCD(K,N)$ space 
${\sf X}$ is non-collapsed if $\m$ is the $N$-dimensional Hausdorff measure. In this case, it is known that $\mathcal{R}_N$ has full measure.
 The following theorem was proven by Kapovitch and Mondino.

\begin{theorem}[\cite{kapovitch_mondino}]\label{thm-MK}
Let $K\in \mathbb R$, $N\in \mathbb N$ and $\alpha\in (0,1)$. There exists $\bar \epsilon=\bar \epsilon(K,N, \alpha)>0$ such that for any $\epsilon\in (0,\bar \epsilon]$ there exists $\bar r=\bar r(K,N,\alpha, \epsilon)>0$ that satisfies the following. 

Let $\sf X$ be a non-collapsed $\RCD(K,N)$ space and let $x\in (\mathcal R_N)_{\epsilon,r}$ for some $r\in (0,\bar r)$. Then there exists a topological embedding
$$F: B_{\alpha r}(0)\subset \mathbb R^N\rightarrow X$$ 
such that $F(B_{\alpha r}(0))\supset B_{\alpha r}(x)$, and the maps $F, F^{-1}$ are $\alpha$-H\"older continuous and $\Psi(\epsilon|N)$-Gromov-Hausdorff approximations.
\end{theorem}

\begin{corollary}\label{cor-balls}
Let $\{M_i\}_{i\in \mathbb N}$ be a sequence of Riemannian manifolds with $\ric_{M_i}\geq K$, $\dim_{M_i} = N$ and $\vol_{M_i}(B_1(o_i))\geq v >0$ and let $\sf X$ be an $\RCD(K,N)$ space for $K\in \mathbb R$ and $N\in \mathbb N$ such that $\sf X$ is the Gromov-Hausdorff limit of the $M_i$'s. Let $\alpha\in (0,1)$. Then, there exists $\bar \epsilon= \bar \epsilon(K, N, \alpha)>0$ such that  for any $\epsilon\in (0,\frac{1}{2}\bar\epsilon]$ there exists $\bar r= \bar r(K,N, \alpha, \epsilon)>0$ satisfying  the following assertion: 

Let $x\in (\mathcal R_N)_{\epsilon,r}$ for some $r\in (0, \bar r)$ and let $x_i\in M_i$ such that $x_i\rightarrow x$. Then,  there exists $i_0\in \mathbb N$ such that  for  $i\geq i_0$ there exist topological embeddings 
$$F: B_{\alpha r}(0)\subset \R^N\rightarrow X \quad \text{and} \quad 
F_i: B_{\alpha r}(0)\subset \R^N \rightarrow M_i,$$ 
where $F$ and $F_i$ have the properties as in the previous theorem. In particular,  the images $F(B_{\alpha r}(0))$ and $F_i(B_{\alpha r}(0))$ are open and homeomorphic to the ball $B_{\alpha r}(0)\subset \R^N$, and $F_i(B_{\alpha r}(0)) \supset B_{\alpha r}(x_i)$, $F(B_{\alpha r}(0)) \supset B_{\alpha r}(x)$.
\end{corollary}

\begin{proof} Observe that the non-collapsing assumption and volume continuity ensure that the limit measure is $\mathcal{H}^N$ and $\sf X$ is non-collapsed. Choose $\bar\epsilon>0$  as in Theorem \ref{thm-MK} and let $\epsilon\in (0, \frac{1}{2}\bar \epsilon]$. Then we choose \[\bar r= \min\{\bar r(K,N, \alpha, \epsilon), \bar r(K, N, \alpha, 2 \epsilon)\}\] where the $\bar r$'s that appear inside the minimum are again as in Theorem \ref{thm-MK}. Assume $x\in (\mathcal R_N)_{\epsilon, r}$ for some $r\in (0, \bar r)$. 
Then there exists a topological embedding 
\[F: B_{\alpha r}(0)\subset \mathbb R^N\rightarrow X,\] 
 such that $F(B_{\alpha r}(0))\supset B_{\alpha r}(x)$ as in Theorem \ref{thm-MK}. Now, since $x_i\rightarrow x$, there exists $i_0\in \mathbb N$ such that $x_i\in (\mathcal R_N)_{2\epsilon, r}$ for $i\geq i_0$. Hence, by Theorem \ref{thm-MK}, there exists a topological embedding 
$$F_i: B_{\alpha r}(0) \subset \mathbb R^N \rightarrow M_i,$$ 
such that $F_i(B_r(0))\supset B_{\alpha r}(x_i)$ and $F_i, F_i^{-1}$ are $\alpha$-H\"older continuous.  By the invariance of the domain theorem,  $F(B_{\alpha r}(0))$ and $F_i(B_{\alpha r}(0))$ are open and,
$F: B_{\alpha r}(0)\rightarrow F(B_{\alpha r}(0))$, 
$F_i: B_{\alpha r}(0)\rightarrow F_i(B_{\alpha r}(0))$ are homeomorphisms.
\end{proof}

\begin{corollary}\label{cor-coverings}
Let $\{M_i\}_{i\in \mathbb N}$ be a sequence of Riemannian manifolds with $\ric_{M_i}\geq K$, $\dim_{M_i} = N$ and $\vol_{M_i}(B_1(o_i))\geq v >0$ and let $\sf X$ be an $\RCD(K,N)$ space for $K\in \mathbb R$ and $N\in \mathbb N$ such that $\sf X$ is the Gromov-Hausdorff limit of the $M_i$'s. Let $\alpha\in (0,1)$ and $\eta>0$.
Let $\bar \epsilon= \bar \epsilon(K, N, \alpha)>0, \epsilon \in (0, \frac{\bar \epsilon}{2}]$ and $\bar r= \bar r(K, N, \alpha, \epsilon)$ as in Corollary \ref{cor-balls}. Then there exist $x^1, \dots, x^k\in X$ and $r^1, \dots, r^k\in (0, \bar r)$  such that the closed balls $\overline{B_{\alpha r^l}(x^l)}$, $l=1, \dots, k$, are mutually disjoint. Moreover there exist $i_1 \in \mathbb{N}$ such that for all $i \geq i_1$ and for all $l\in \{1, \ldots, k \}$ there exist $x_i^l\in M_i$ such that $x_i^l \rightarrow x^l$ and 
\begin{itemize}
\item $\overline{B_{\alpha r^l}(x^l_i)}$ converges in GH sense to $\overline{B_{\alpha r^l}(x^l)}$
    \item the closed balls $\overline{B_{\alpha r^l}(x^l_i)}$, $l=1, \dots, k$, in $M_i$ are disjoint,
    \item $$\vol_{M_i}\left(M_i\backslash \bigcup_{l=1}^k B_{\alpha r^l}(x_i^l)\right)< 2\eta,$$
    \item  for large $i$, there exist topological embeddings 
    \[
    F^l_i: B_{\alpha r^l}(0)\rightarrow M_i\qquad F^l_i(B_{\alpha r^l}(0))\supset B_{\alpha r^l}(x^l_i)
    \] 
     and $F_i^l, (F_i^l)^{-1}$ are $\alpha$-H\"older continuous. 
\end{itemize}
In particular, the open sets $U^l_i:= F^l_i(B_{\alpha r^l}(0))\subset M_i$ are simply connected. 
\end{corollary}

\begin{proof}
Recall that by definition of $\mathcal{R}_N$ and by the fact that $\sf X$ is non-collapsed we have
\[\mathcal{R}_N=\bigcap_{\epsilon>0}\bigcup_{r>0}(\mathcal{R}_N)_{\epsilon,r},\quad \mathcal{H}^N(X \setminus \mathcal{R}_N)=0.\]
Let $\bar{\epsilon}=\bar \epsilon(K,N, \alpha), \epsilon$ and $\bar{r}$ be as in the previous corollary. 
Hence, for any $\epsilon\in (0, \frac{\bar \epsilon}{2}]$, we have 
 that for any $x \in \mathcal{R}_N$ there exists $r_{x}>0$ such that $x \in (\mathcal{R}_N)_{\frac{\epsilon}{2},{\Large r}}$  for all $r\in (0, r_x]$. Moreover, by definition of $(\mathcal R_N)_{\frac{\epsilon}{2}, r_x}$ we can assume $r_x\in (0, \bar r)$. 

Consider the covering of $\mathcal{R}_N$ given by the balls $B_{\alpha r}(x)$, where $x \in \mathcal{R}_N$ and $r\in (0,r_x]$.  This is a Vitali covering of $X$.  Let then $\eta>0$ be fixed.  Vitali's covering theorem implies that there exist $x^1, \dots, x^k\in X $ and $r^1, \dots, r^k\in (0, \bar r)$  such that
\begin{itemize}
\item $x^l\in (\mathcal R_N)_{\frac{\epsilon}{2}, r^l}$ for every $l\in \{1, \dots, k\}$,
\item the closed balls $\overline{B_{\alpha r^l}(x^l)}$, $l=1, \dots, k$, are mutually disjoint, 
\item $\mathcal H^N\left(X\backslash \bigcup_{l=1}^k B_{\alpha r^l}(x^l)\right)< \eta$.
\end{itemize}
Moreover, according to Theorem \ref{thm-MK}, for any $l\in \{1, \ldots, k\}$ there exists a topological embedding $F^l: B_{\alpha r^l}(0)\rightarrow {X}$.

For any $l \in \{1, \ldots, k\}$, let $x_i^l\in M_i$, $x_i^l \to x^l$ for $i \to \infty$. In particular, there exists $i_0$ such that for all $i \geq 0$ the point $x_i^l$ belongs to the $(N, \epsilon, r^l)$-regular set of $M_i$. Moreover, for any $l \in \{1, \ldots, k\}$ the balls $\overline{B_{\alpha r^l}}(x_i^l)$ converge in GH sense to $\overline{B_{\alpha r^l}}(x^l)$. By our choices of $\epsilon$ and $r^l \in (0, \bar{r})$, we can apply Corollary \ref{cor-balls}: there exists $i_1 \in \mathbb{N}$, $i_1\geq i_0$ such that for all $i \geq i_1$ for $l=\{1, \ldots, k\}$ the closed balls $\overline{B_{\alpha r^l}}(x_i^l)$ are disjoint, we have 
$$\vol_{M_i}\left(M_i\backslash \bigcup_{l=1}^k B_{\alpha r^l}(x_i^l)\right)< 2\eta,$$
and there exist topological embeddings $F^l_i: B_{\alpha r^l}(0)\rightarrow M_i$ such that $F^l_i(B_{\alpha r^l}(0))\supset B_{\alpha r^l}(x^l_i)$ is open and $F_i^l, (F_i^l)^{-1}$ are $\alpha$-H\"older continuous.  In particular, the open set $F^l_i(B_{\alpha r^l}(0))=: U^l_i\subset M_i$ is simply connected. 
\end{proof}

\begin{proof}[Proof of Theorem \ref{th:harmonicVF}] 
We divide the proof of the theorem in three parts.

{\bf 1. Construction of harmonic functions on each $U^l_i \subset M_i$:}
Since $\vol_{M_i}\geq v$, the sequence $M_i$ is non-collapsed and $\sf X$ is a non-collapsed $\RCD(K,N)$ space, that is $\m={ c}\mathcal H^N$ { for a normalization constant $c>0$}.

Let $\alpha, \bar \epsilon$ and $\bar r$, as well as $x_i^l$, $l=1, \dots, k$ and $U_i^l=F^l_i(B_{\alpha r^l}(0))$ be as in Corollary \ref{cor-coverings}. We know that 
$$\vol_{M_i}\left(M_i\backslash \bigcup_{l=1}^k B_{\frac{1}{2}\alpha r^l}(x_i^l)\right)< 2\eta.$$
Since $U_i^l$ is simply connected for every $(l,i)\in \{1, \dots, k\}\times \N$, by the Poincar\'e lemma we can find a function $f_i^l$ such that $\nabla f^l_i= V_i$ on $U^l_i$. Moreover,
\[\Delta_{M_i} f_i^l= {\rm d}^{\star} {\rm d}{f_i^l} = {\rm d}^{\star} V_i^{\#}=0\]
since $V_i$ is harmonic. Hence $f_i^l$ is harmonic on $U_i^l$ and, by locality of the Laplace operator, it is also harmonic on $B_{\alpha r^l}(x_i^l)\subset U_i^l$.\\
\\
{\bf 2. Obtaining $V$ and showing it is locally the gradient of harmonic functions:}
By adding constants we can assume that $\int_{B_{\alpha r^l}(x_i^l) } f_i^l \, {\rm d}\vol_{M_i}=0$. Moreover, we also have that $\int_{U_i^l} |\nabla f_i^l|^2 \, {\rm d}\vol_{M_i}\leq  C$ because the harmonic functions $f_i^l$  locally satisfy a uniform gradient estimate with constant $C>0$ that only depends on $K, N$ and $r^l>0$.

We can apply Theorem \ref{th:211} for every $l=1, \dots, k$.  From this we conclude that there exist harmonic functions $f^l$ on $B_{ \frac{1}{2}\alpha r^l}(x^l)\subset X$ such that, possibly extracting suitable  subsequences, $f_i^l\rightarrow f^l$ strongly in $W^{1,2}(B_{ \frac{1}{2} \alpha r^l}(x^l))$.

On the other hand $V_i$ is uniformly $L^2$-bounded by the constant $C$. Therefore, we can extract another subsequence such that $V_i\rightarrow V$ $L^2$-weakly for a vector field $V$ on $X$.  By definition of $L^2$-weak convergence it follows that  $(\langle V_i, \nabla h_i\rangle)_{i\in \N}$  converges $L^2$ weakly for any sequence $(h_i)$ that converges strongly in $W^{1,2}$ to $h$ and such that $\sup_{i} \left\||\nabla h_i|\right\|_{L^\infty({\vol_{M_i}})}{\color{red} \leq C< \infty}$.  

Now, by Corollary \ref{cor:pairing}  it also follows for such a sequence $h_i$ with support in $B_{\frac{1}{2}\alpha r^l}(x_i^l)$ that $\langle \nabla f_i^l, \nabla h_i\rangle$ converges weakly to $\langle \nabla f^l , \nabla h\rangle$. 

Hence 
$V= \nabla f^l$ on $B_{ \frac{1}{2} \alpha r^l}(x^l)$. 
\\
\\
{\bf 3. $L^2$-strong convergence of $V_i$ to $V$:}
The standard Bochner formula holds for $V_i$ on $M_i$, i.e. 
$$\Delta_{M_i} |V_i|^2= |\nabla V_i|^2 + \ric_{M_i}(V_i,V_i)\geq K |V_i|^2.$$
Hence elliptic regularity theory for $|V_i|^2$ on $M_i$ yields that $$\left\| |V_i|^2\right\|_{L^\infty(M_i)} \leq \tilde C(K, N, C, D)=: \tilde C.$$

The final step consists in showing the $L^2$ strong convergence of $V_i$ to $V$. For that purpose we observe that
\begin{align*}
\int_{M_i} |V_i|^2 \, {\rm d}\!\vol_{M_i} &=  \sum_{l=1}^{k} \int_{B_{ \frac{1}{2} \alpha r^l}(x_i^l)} |V_i|^2 \, {\rm d}\!\vol_{M_i} + \int_{M_i\backslash \bigcup_{l=1}^kB_{ \frac{1}{2} \alpha r^l}(x_i^l)} |V_i|^2 \, {\rm d}\!\vol_{M_i}\\
& \leq  \sum_{l=1}^{k} \int_{B_{ \frac{1}{2} \alpha r^l}(x_i^l)} |\nabla f_i^l|^2 \, {\rm d}\!\vol_{M_i}+ \tilde C\eta\\
&\rightarrow  \sum_{l=1}^{k} \int_{B_{ \frac{1}{2} \alpha r^l}(x^l)} |\nabla f^l|^2 \, {\rm d}\!\haus^N  + \tilde C\eta \\
& \leq  \int_X |V|^2 \, {\rm d}\!\haus^N+ {\tilde C} \eta.
\end{align*}
Since $\eta>0$ was arbitrary, we get 
\[\limsup_{i\rightarrow \infty} \int_{M_i} |V_i|^2 \, {\rm d}\!\vol_{M_i} \leq \int_X |V|^2 \, {\rm d}\!\haus^N.\]
This gives $L^2$-strong convergence of $V_i$ to $V$.
\end{proof}

Let us end this section by presenting the following result, which will play a key role in the proof of Theorem \ref{mthm0}. Although (2) will not be relevant elsewhere in the paper, we add it due to its independent interest. 

\begin{proposition}\label{approx}
Let $M_i$ be a Gromov-Hausdorff convergent sequence of closed Riemannian manifolds of dimension $n$ with $\ric_{M_i}\ge K$ for some $K \in \mathbb{R}$ to a closed Riemannian manifold $M$ of the same dimension. Let $F_i:M_i \to M$ be a sequence of equi-regular 
maps in the following sense: for any $f \in C^{\infty}(M)$ we have $f \circ F_i \in D(\Delta_{M_i})$ with
\begin{equation}
    \sup_i\|\Delta_{M_i}(f \circ F_i)\|_{L^{\infty}}<\infty.
\end{equation}
Then the following two conditions are equivalent:
\begin{enumerate}
\item[(a)] $F_i$ is a $\delta_i$-Gromov-Hausdorff approximation for some $\delta_i \to 0^+$;
\item[(b)] we have $\|F_i^*g_M-g_{M_i}\|_{L^p} \to 0$ for any $1 \le p < \infty$, namely $F_i^*g_M$ $L^p$-strongly converge to $g_M$.
\end{enumerate}
Moreover, if one of the above holds, then:
\begin{enumerate}
\item for all $x, y \in M_i$ and for some $\epsilon_i \to 0^+$, we have
\begin{equation}\label{lip}
(1-\epsilon_i)\de_{M_i}(x, y)^{1+\epsilon_i} \le \de_M(F_i(x), F_i(y))\le C(n) \de_{M_i}(x, y).
\end{equation}
In particular $F_i$ is a diffeomorphism for any sufficiently large $i$;
\item For any smooth closed $1$-form $\omega$ on $M$,  the harmonic representatives of smooth closed $1$-forms $F_i^*\omega$ on $M_i$ in $H^1_{\mathrm{dR}}(M_i;\mathbb{R})$ $L^2$-strongly converge to the harmonic representative of $\omega$ in $H^1_{\mathrm{dR}}(M;\mathbb{R})$.
\end{enumerate}
\end{proposition}

\begin{proof}
Firstly let us remark that the equi-regularity 
for the equi-regular maps $F_i$ 
implies the equi-Lipschitz continuity of $F_i$ because of the following: Fix a Riemannian isometric embedding $\Phi:M \hookrightarrow \mathbb{R}^l$. Applying \cite[Theorem 3.1]{Jiang} to $\Phi \circ F_i$ shows that $\Phi \circ F_i$ is equi-Lipschitz.
Then recalling that $\Phi$ is bi-Lipschitz onto its image, we see that $F_i=\Phi^{-1} \circ (\Phi \circ F_i)$ is equi-Lipschitz.

Thus in the sequel, after passing to a subsequence, with no loss of generality we can assume that $F_i$ converge uniformly to a Lipschitz map $F:M \to M$. 

    Assume (a). Then $F$ must be an isometry {of metric measure spaces}, in particular $F$ is smooth {because it is well-known that a map between Riemannian manifolds is a Riemannian isometry if and only if it is an isometry of metric spaces, see for instance \cite[Theorem 11.1]{Hel}}. 
    On the other hand, applying  \cite[Theorem 4.4]{ambrosio_honda_local}, after passing to a subsequence again, $\phi_j \circ F_i$ converge -strongly in $W^{1,2}$ to some $\bar \phi_j \in D(\Delta_M)$ with $\Delta_M \bar \phi_j \in L^{\infty}(M)$. Thus letting $\bar \Phi=(\bar \phi_i)_{i=1}^l$, we have $\Phi \circ F= \bar \Phi$. In particular $\bar \Phi$ is also a smooth isometric embedding. This observation allows us to conclude that $F_i^*g_M=(\Phi \circ F_i)^*g_{\mathbb{R}^l}$ $L^p$-strongly converge to $\bar \Phi^*g_M=g_M$ for any $1 \le p< \infty$. Since $g_{M_i}$ also $L^p$-strongly converge to $g_M$ (see \cite[Proposition 3.78]{Honda15}), we have (b).

    Assume (b). As observed above, with no loss of generality we can assume that $\phi_j \circ F_i$ $W^{1,2}$-strongly converge to some $\bar \phi_j \in D(\Delta_M)$ with $\Delta_M \bar \phi_j \in L^{\infty}(M)$. Then the elliptic regularity theory proves that $\bar \Phi$ is $C^1$, where $\bar \Phi=(\bar \phi_i)_{i=1}^l$.  On the other hand, our assumption (b) shows that $\bar \Phi^*g_{\mathbb{R}^l}$ coincides with the $L^p$-limit of $(\Phi \circ F_i)^*g_{\mathbb{R}^l}=F_i^*g_M$, namely $g_M$. Therefore $\bar \Phi$ preserve the length of any  smooth curve. Since $\Phi \circ F=\bar \Phi$, then $F=\Phi^{-1} \circ \bar \Phi$ also preserves the length of any smooth curve. Combining this with the compactness of $M$ implies that $F$ is an isometry. This easily implies (a).
    
    Finally let us prove the remaining statements.
    (1) is already proven in \cite[Theorem 1.2]{HondaPeng}.
    Thus it is enough to prove (2). The proof is divided into 3 steps as follows.

    \textbf{Step 1}: {\it Claim:} $f \circ F_i$ $W^{1,2}$-strongly converge to $f$ for any $f \in C^{\infty}(M)$. 
    
    This is a direct consequence of \cite[Corollary 1.10.4]{ambrosio_honda} (or \cite[Theorem 4.4]{ambrosio_honda_local}) with (a). 
    
    \textbf{Step 2}: {\it Claim:} For any smooth $1$-form $\omega$ on $M$, $F_i^*\omega$ $L^2$-strongly converge to $\omega$. 
    
    The proof is as follows. The desired result is correct in the case when $\omega=fdg$ for some $f, g \in C^{\infty}(M)$ because of \textbf{Step 1}. Since any smooth $1$-form on $M$ can be written as a linear combination of forms $fdg$, we conclude.

    \textbf{Step 3}: Denote by $\eta_i$ the harmonic representative of the closed $1$-form $F_i^*{\omega}$ on $M_i$. Thus we can find $f_i \in C^{\infty}(M_i)$ with $\int_{M_i}f_i\de\!\vol_{M_i}=0$ and 
    \begin{equation}\label{198}
    F_i^*{\omega}=\eta_i+df_i.
    \end{equation}
    Since \begin{equation}\label{199}
    \|F_i^*{\omega}\|_{L^2}^2=\|\eta_i\|_{L^2}^2+\|df_i\|_{L^2}^2,
    \end{equation}
    {after passing to a subsequence, we can find; 
    \begin{itemize}
    \item by Theorem \ref{th:211}, the $W^{1,2}$-weak limit $f \in W^{1,2}(M)$ of $f_i$;
    \item by Theorem \ref{th:harmonicVF}, the $L^2$-strong limit $\eta \in W^{1,2}_d(T^*M) \cap D(\de^{\star}) \cap W^{1,2}_C(TM)$ of $\eta_i$ with $d\eta=0$ and $d^{\star}\eta =0$ because of Lemma \ref{lem:div}, Theorems \ref{thm-Honda} and \ref{th:stablity d}.
    \end{itemize} 
    
    Note that weak convergence in the first item can be easily improved to $W^{1,2}$-strong convergence via (\ref{199})).
    }Thus letting $i \to \infty$ in (\ref{198}) yields
    \begin{equation}\label{197}
        \omega=\eta+df.
    \end{equation}
On the other hand, since $M$ is smooth, we have $W^{1,2}=H^{1,2}$ for vector fields. In particular $\eta$ is harmonic in weak sense, thus by elliptic regularity theory, $\eta$ is smooth and harmonic. Therefore $f$ is also smooth because of applying $\de^{\star}$ in both sides of (\ref{197}) with elliptic regularity theory again. Hence $\eta$ is the harmonic representative of $\omega$, which completes the proof.\end{proof}

\begin{remark}\label{approxrem}
Let $M_i$ be a Gromov-Hausdorff convergent sequence of closed Riemannian manifolds of dimension $n$ with $\ric_{M_i}\ge K$ for some $K \in \mathbb{R}$ to a closed Riemannian manifold $M$ of the same dimension. 
It is proved in \cite[Theorem 1.2]{HondaPeng} that for any sufficiently large $i$, there exists a diffeomorphism $F_i:M_i \to M$ such that $F_i$ is a $\delta_i$-Gromov-Hausdorff approximation for some $\delta_i \to 0^+$ and that (\ref{lip}) for this $F_i$ can be improved to 
\begin{equation}\label{lipim}
(1-\epsilon_i)\de_{M_i}(x, y)^{1+\epsilon_i} \le \de_M(F_i(x), F_i(y))\le (1+\epsilon_i) \de_{M_i}(x, y)
\end{equation}
for all $x, y \in M_i$, for some $\epsilon_i \to 0^+$. 
\end{remark}

\section{Torus rigidity}\label{sec:torusrigidity}

The purpose of this section is to establish a rigidity result for tori, which should be compared with the result of Gigli-Rigoni \cite{GR18}. Before presenting the result, we review the definition of analytic dimension.

Let's recall a crucial result from {Bru\`e} and Semola in \cite[Theorem 3.8]{BS18}: if $\sf X$ is an $\RCD(K, N)$ space for some $K \in \mathbb R$ and $1 \le N < + \infty$, then there exists a natural number $1 \le n \le N$ such that the tangent cone of $\sf X$ is the $n$-dimensional Euclidean space at $\m_X$-almost every point in $X$. A direct consequence of this ensures that, in this case, the tangent module $L^2(T \sf X)$ has a constant dimension equal to $n$ (see \cite[Corollary 3.10]{BS18}).

Using the terminology introduced by \cite[Definition 2.10]{Han19}, we refer to this $n$ as the analytic dimension of $\sf X$.

\begin{theorem}\label{thm-torus}
Let $\sf X$ be an $\RCD(K,N)$ space with $K\in \mathbb R$ and $N\in (1, \infty)$. Let $n$ be the analytic dimension of  $\sf X$. Moreover, let $V_1, \dots, V_n$ be  $L^2$-orthogonal vector fields  in
$H_C^{1,2}(T {\sf X}) \cap D(\Div)$  with $\nabla V_i=0$ and $\Div V_i = 0$ for $i = 1, \dots, n$. 
Then  $\sf X$ is isomorphic to a flat $n$-torus. 
\end{theorem}

\begin{remark}\label{rmk-torus} In Theorem \ref{thm-torus}, if we  assume that $\sf X$ is non-collapsed, then $\mathrm{div} (V_i)=0$ can be dropped because 
of Lemma \ref{h=c}.
\end{remark}

The validity of Theorem \ref{thm-torus} follows by applying the same strategy as the proof presented in detail in \cite[Section 4]{GR18}. However, it is important to underline that the setting of \cite{GR18} is that of $\RCD(0, N)$ spaces. In this more general framework, where the lower bound on the curvature is no longer 0, we need to carefully modify certain arguments from the proof in  \cite{GR18}.

Below, we summarize the main results leading to the proof of Theorem \ref{thm-torus}. We pay special attention to the points where a vanishing covariant derivative or a divergence-free vector field hypothesis is needed. In particular, we provide an analogy with the smooth case, showing that if a vector field $V$ defined on an $\RCD(K, \infty)$ space satisfies $\Div V = 0$ and $\nabla {V} = 0$, then the RLF, denoted as $F^t$, associated with it is a measure-preserving isometry. This proposes a different argument compared to the one in \cite{GR18}.

\begin{remark}\label{rem:WH}
    We point out that the assumption in Theorem \ref{thm-torus} for the vector fields $V_i$'s to be in the space $H^{1, 2}_C(T \sf X)$, and not just in $W^{1, 2}_C(T \sf X)$, is actually needed. In fact, this is essential to ensure that the scalar product $\langle V_i, V_j \rangle$ is in $W^{1,1}(\sf X)$, as pointed out in \cite[Proposition 3.4.6]{giglinonsmooth}. In particular, if the vector fields are harmonic with $\nabla V=0$ or $K=0$, then they satisfy the hypothesis of Theorem \ref{thm-torus}.
     
     It is also natural to ask whether  a non-collapsed $\RCD(K, N)$ space $\sf X$ with $V_i \in W^{1,2}_C(T \sf X)$, $i=1,2,\ldots, N$, satisfying $\nabla V_i=0$ and $\int_X\langle V_i, V_j\rangle \de\!\m_X=\delta_{ij} $, must be isomorphic to a flat torus. However, this question has a negative answer. A closed interval $[0, \pi]$ endowed with $1$-dimensional Hausdorff measure  and the parallel vector field $\nabla x \in W^{1,2}_C$ provides a counterexample ($\nabla x \notin H^{1,2}_C$ due to the boundary condition). 
\end{remark}

\subsection{Behavior of the $W_2$ distance under the RLF of a vector field}

{Notice that, by a slight abuse of notation, in this section we sometimes use $T$ to denote the time of existence of the Regular Lagrangian Flows.}

\begin{proposition}\label{prop:boundW2} 
Let $V(t, \cdot ) \in W^{1,2}_{C,loc}(T{\sf X})$  for any $t \ge 0$ with $\left\||V|\right\|_{L^{\infty}}\leq C<\infty$ and for which there exists a unique RLF, $F\colon [0, \infty) \times {\sf X} \rightarrow {\sf X}$.   Let $x_0\in X$ and $\mu_0, \mu_1 \in \mathcal P(X, \m_{{X}})$ such that $\mu_i(B_R(x_0))=1$  for some $R>0$. Define $\mu_i^t =(F^t)_{\#}\mu_i \in \mathcal P(X, \m_{{X}})$, $i=0,1$. If we further assume that $V$ is infinitesimally $k$-monotone in $X$, then 
\begin{align}\label{eq:contW2}
W_2(\mu_0^t, \mu_1^t)^2 \leq e^{-2kt} W_2(\mu_0, \mu_1)^2 \qquad \forall t\in (0,T).
\end{align}
\end{proposition}

We  follow the proof of $(1) \Rightarrow (2)$ of  \cite[Theorem 3.18]{han18}. We outline the main steps of the proof for completeness.

\begin{proof}  
We consider $\mu_0, \mu_1\in \mathcal P(X, \m_X)$ concentrated in $B_{R}(x_0)$  for some $R>0$ and $(\mu^t_i)_{t\in (0, T]}\in \mathcal P(X, \m_X)$ be the evolution of $\mu_i$ under the RLF. 
\\

\smallskip

{\it Claim:} For $s\in (0,T)$ fixed it holds 
\begin{align}\label{id:derivative}\frac{d}{dt} \frac{1}{2}W_2(\mu_0^t, \mu_1^s)^2 = -\int_X \langle V, \nabla \phi^{t,s}\rangle \de\!\mu^t_0 \ \mbox{ for a.e. } t\in [0,T],
\end{align}
where $\phi^{t,s}$ is a Kantorovich potential relative to $(\mu_0^t, \mu_1^s)$.  Since $\mu^t_0$ and $\mu^s_1$ have bounded support, we can choose the Kantorovich potential to be Lipschitz. 

The claim then follows from \cite[Theorem 4.6]{GigliHan} since $(\mu^t_0)_{t\in [0,T]}$ solves the continuity equation with respect to the vector field $V$. 

\smallskip

Now, the function $t\in [0,T] \mapsto \int_X \langle V, \nabla \phi^{t,s}\rangle d{\mu_0^t}$ is continuous.  Hence, it follows that \eqref{id:derivative} holds for every $t\in [0,T]$, $t\in (0,T) \mapsto \frac{1}{2} W_2(\mu_0^t, \mu_1^s)^2$ is  $C^1$ and in particular, we can take the derivative at $t=s$. Pick $t\in (0,T)$ where $W_2(\mu^t_0, \mu_1^t)^2$ is differentiable. 
Then 
\begin{align*}
&\frac{d}{dt} W_2(\mu_0^t, \mu_1^t)^2=\lim_{h\downarrow 0} \frac{W_2(\mu_0^{t+h}, \mu_1^{t+h})^2 - W_2(\mu_0^t, \mu_1^t)^2}{h}\\
&= \lim_{h\downarrow 0}\frac{W_2(\mu_0^{t+h}, \mu_1^{t+h})^2-W_2(\mu_0^{t+h}, \mu_1^t)^2 + W_2(\mu_0^{t+h}, \mu_1^t)^2 - W_2(\mu_0^t, \mu_1^t)^2}{h}\\
&= \frac{d}{dr}\Big|_{r=t} W_2(\mu_0^t, \mu_1^r)^2 + \frac{d}{ds}\Big|_{s=t} W_2(\mu_0^s, \mu_1^t)^2.
\end{align*}

In combination with \eqref{id:derivative} at $r=t$ and $s=t$ and $k$-monotonicity 
(recall that by Remark \ref{rem-monotoneiff}, if $V$ is infinitesimally $k$-monotone in $X$, then $V$ is $k$-monotone in $X$) this yields
\[\frac{\de}{\de\!t} W_2(\mu_0^t, \mu_1^t)^2\leq - 2k W_2({\mu_0^t}, {\mu_1^t})^2 \mbox{ for a.e. $t\in (0,T)$}.\]
Applying Gr\"onwall's lemma finishes the proof of the proposition. 
\end{proof}

We will now prove that, under mild assumptions on the vector field $V$, the RLF $F^t_V$ associated with $V$ and $F^t_{-V}$, the one associated with $-V$, are inverses of each other. We follow the proof in \cite[Lemma 3.18]{GR18}, where the vector field $V$ is defined on an $\RCD(0, \infty)$ space and is assumed to be harmonic. Before proceeding with the proof, let us recall the following fact (see \cite[Section 3.3]{GR18}):
\begin{center}
    If $S', S \colon {\sf X} \to {\sf X}$ are two Borel maps such that\\ 
    $S'_\ast \mu = S_\ast \mu$  for all $\mu \in \mathcal P(X, \m)$ with bounded support and bounded density, \\
    then $S'= S$ $\m$-a.e.
\end{center}

\begin{lemma}
Let $V \colon [0,T] \rightarrow L^2_{loc}(T{\sf X})$ be a Borel vector field such that $V_t\in D(\Div)$ for every $t\in [0,T]$. Assume that $\left\||V_t|\right\|_{L^2(\m)} \in L^1(0,T), \left\|\Div(V_t)\right\|_{L^\infty}\in L^\infty([0,T])$ and $\left\|\nabla_{sym} V_t\right\|_{L(T^{\otimes 2}X)}\in L^1([0,T])$. Then for every $t \ge 0$ we have the validity $\m$-a.e. of the following identities
\[
F^t_{-V} \circ F^t_{V} = \text{Id} \quad \text{ and } \quad F^t_{V} \circ F^t_{-V} = \text{Id}.
\]
\end{lemma}

\begin{proof}
First observe that the assumptions of the lemma are such that the vector field $V \colon [0, T] \to L^2_{\text{loc}}(T {\sf X})$ satisfies the hypotheses of Theorem \ref{thm:RLF} as it also does the vector field $-V$.  Now, without loss of generality, we prove the identities for $t = 1$. Let $\mu \in \mathcal P(X, \m_X)$ be with bounded support and bounded density and consider the two curves \[[0, 1] \ni t \mapsto \mu_t, \tilde \mu_t \in \mathcal P(X, \m_X)\]
defined as
\[
\mu_t:= \big(F^{1-t}_V\big)_\sharp \mu \quad \text{ and } \quad \tilde \mu_t := \big( F^t_{-V}\big)_\sharp \big( F^1_V\big)_\sharp \mu.
\]
Notice that $\mu_0  = \tilde \mu_0$ and that both curves solve the continuity equation in the sense of \eqref{eq:CE} associated with the vector field $-V$. The uniqueness of the RLF expressed in Theorem \ref{thm:RLF}-$(iv)$ ensures that $\mu_1 = \tilde \mu_1$, namely
\[\mu = \big(F^1_{-V} \circ F^1_V \big)_\sharp \mu.\]
We can then conclude by the arbitrariness of the measure $\mu$.
\end{proof}

\subsection{Behavior of distance under the flow of a parallel and divergence-free vector field}

Next, we improve the results of the previous subsection by assuming that the vector field $V \colon [0, T] \to L^2_{\text{loc}}(T X)$ has a vanishing covariant derivative. We then assume that $V$ is divergence-free and conclude that its RLF has a measure-preserving continuous representative. We conclude this section by proving Theorem \ref{thm-torus}.

\begin{proposition}\label{prop:isommu}
Under the hypotheses of Proposition \ref{prop:boundW2}, additionally assume that $\nabla V\equiv 0$ and $|| \Div(V_t) ||_{L^\infty} \in L^\infty([0, T])$ on $X$. Then for $\mu_i$ supported in $B_R(x_0)$ and $(\mu_i^t)_{ t\in (0,T)}$ as before it follows that 
\begin{equation}
\label{eq:isometryW2}
W_2(\mu_0^t, \mu_1^t)= W_2(\mu_0, \mu_1) \qquad \forall t\in (0,T).
\end{equation}
\end{proposition}

\begin{proof} 
By Proposition \ref{prop:boundW2} we have the inequality \[W_2(\mu_0^t, \mu_1^t) \leq W_2(\mu_0, \mu_1) \qquad \forall t\in {(0, T)},\]
since $\nabla V\equiv 0$ implies infinitesimal $0$-monotonicity, as shown in Corollary \ref{cor:monV}. 

We observe that the only inequality that enters the proof of the previous proposition comes from the infinitesimal $k$-monotonicity that also holds  in the other direction by assumption.
\end{proof}

With the previous proposition we can now follow the strategy proposed to prove $(3) \Rightarrow (4)$ in \cite[Theorem 3.16]{han18} to show that the RLF has a continuous representative which is also an isometry.

\begin{proposition}\label{prop:isom}
    Assume that $\nabla V \equiv 0$ and $|| \Div(V_t) ||_{L^\infty} \in L^\infty([0, T])$ on $X$. Then for every $t \in \R$ the RLF $F^t$ associated to $V$ has a unique continuous representative which is an isometry, namely
    \begin{equation}\label{eq:isometry}
    \de(F^t(x), F^t(y)) = \de (x, y) \qquad \forall \, x, y \in X.
    \end{equation}
\end{proposition}

\begin{proof}
    Fix an arbitrary point $x_0 \in X$. Let us start by proving the following:
    
\emph{Claim:} The curve of measures $(F^t)_{\sharp} \delta_{x_0}$, starting at $\delta_{x_0} \in \mathcal P_2(X)$, is uniquely defined and supported on a curve in $X$. 

\emph{Proof of the claim}: Let $\{\mu_n\}_{n \in \N} \subset \mathcal P_2(X, \m_X)$ be a sequence of measures supported in $B_R(x_0)$ for some $R > 0$ such that $\lim_{n \to \infty} W_2(\mu_n, \delta_{x_0}) = 0$. By \eqref{eq:contW2} we get that the flows $\mu_n^t = (F^t)_{\sharp}\mu_n$  starting from $\mu_n$ are uniformly converging to a curve as $n \to \infty$. We denote this curve by $\big(U_t({x_0})\big)_t \subset \mathcal P_2(X, \m_X)$. We remark that $\big(U_t({x_0})\big)_t$ is independent of the choice of the approximating sequence $\{\mu_n\}_{n \in \N}$. 

We now show that $U_t(x_0)$ is supported on a single point in $X$. Assume that there exists $\bar t > 0$ for which $\text{supp}(U_{\bar t}(x_0))$ contains at least two distinct points $a, b\in X$.
Let $\Pi^n \in \mathcal P(C([0, \infty), X))$ be the lifting of the curve $(F^t)_\sharp \big(\m_X(B_{1/n}(x_0))^{-1} \m_X |_{B_{1/n}(x_0)} \big)$. The uniqueness of the RLF (see $(iv)$ in Theorem \ref{thm:RLF}) ensures the existence of two sets $\Gamma^1_{ n}, \Gamma^2_{n} \subset \text{supp}(\Pi^n)$ with the property that $\m_X(\Gamma^1_{n}), \m_X(\Gamma^2_{n}) > 0$ and that 
\[
\inf\{ \de(\gamma_{\bar t}^1, \gamma^2_{\bar t}) : \gamma^1 \in \Gamma^1_{n}, \gamma^2 \in \Gamma^2_{n} \} > \frac12 \de(a, b) > 0\]
for $n$ large enough. Thus the two sequences of curves $\mu_{n}^{t, i} := (e_t)_\sharp \big(\Pi^n(\Gamma^i_n)^{-1} \Pi^n |_{\Gamma^i_n} \big)$, $i = 1, 2$ are such that $\mu_{n}^{t, i} \to \delta_{x_0}$, but $\mu_{n}^{\bar t, 1} \neq \mu_{n}^{\bar t, 2}$, contradicting the uniqueness of $U_t(x_0)$.

From the above construction, it follows that $U_t(x) = F^t(x)$ for any point $x \in X$ where the curve $(F^t(x))_t$ is well-defined.

Now we have just to see that the above property holds for any $x \in X$ and to do so we extend $F^t$ to the whole space in the following way: for any $x \in X$, we set $(F^t)_\sharp \delta_x = U_t(x) =\delta_{F^t(x)}$.

Finally, the identity in \eqref{eq:isometry} follows immediately from Proposition \ref{prop:isommu}
 
by considering $\mu_0 = \delta_x$ and $\mu_1 = \delta_y$.
\end{proof}

From now on we will also assume that $V$ is a divergence-free vector field. In this setting the following holds (see \cite[Proposition 3.19]{GR18}).

\begin{proposition}[Preservation of the measure]\label{prop-presMeas}
Let $V \colon [0,T] \rightarrow L^2_{loc}(TX)$ be a vector field for which there exists a unique RLF and such that $\Div(V_t) \equiv 0$. Then for every $t \in \R$ it holds
\[
(F^t)_\sharp \m_X = \m_X.
\]
\end{proposition}

By Proposition \ref{prop-presMeas} and Proposition \ref{prop:isom}, the following holds.

\begin{proposition}\label{prop-IsomMeasPres}
    Let $V \colon [0,T] \rightarrow L^2_{loc}(TX)$ be such that $\Div V = 0$ and $\nabla V  = 0$. Then for every $t \in \R$ the map $F^t$ has a continuous representative which  is a measure preserving isometry.
\end{proposition}

Note that the above result is the analogous of \cite[Theorem 3.22]{GR18} in the case in which the vector field $V$ defined on the $\RCD(K, N)$ space $\sf X$.

We proceed to the proof of Theorem \ref{thm-torus}. Let us observe that the condition $\nabla V_i \equiv 0$ for every $i \in \{1, \dots, n\}$ ensures that $\langle V_i, V_j \rangle \in W^{1, 2}(\X)$ with
\[
\de \langle V_i, V_j \rangle = \nabla V_i(\cdot, V_j) + \nabla V_j(\cdot, V_i) = 0 \quad \m\text{-a.e.}
\]
This, in particular, implies that $\langle V_i, V_j \rangle$ is $\m$-a.e. equal to a constant function and, since $\int_X \langle V_i, V_j \rangle \, \de \m_X = 0$ if $i \neq j$, we conclude that $V_i$'s are point-wise orthogonal. Without loss of generality, we can assume that $|V_i|\equiv 1$ $\m$-a.e. for every $i \in \{1, \dots, n\}$. We also remark that, since these vector fields are in $L^2(T {\sf X})$, it holds
\begin{equation}\label{eq:mfinite}
\m_X(X) < \infty.
\end{equation}

\begin{proof}[Proof of Theorem \ref{thm-torus}]
First of all, we recall that if $V, W \in L^2(T \X)$ are two harmonic vector fields, then for any $t, s \in \R$ it holds
\begin{equation}\label{eq:FlVW}
F_V^t \circ F_W^s = F_{tV + s W}^1.
\end{equation}
We refer to \cite[Theorem 3.24]{GR18} for a proof of this identity (notice that in \cite{GR18} this result is formulated in the setting of $\RCD(0, \infty)$ spaces, but the same proof holds for any $\RCD(K, \infty)$ space, $K \in \R$).

By Proposition \ref{prop-IsomMeasPres}, there exist continuous RLF's $F_{V_i}$ of  $V_i$, $i=1,\dots,n$ such that $F^t_{V_i}$ is an isomorphism for all $t$. Let us define the map ${\mathcal F} \colon \X \times \R^n \to \X$ by setting
\[
{\mathcal F}(x, \underbar{a})= F^{a_1}_{V_1} \circ \cdots \circ F^{a_n}_{V_n}(x),
\]
where $\underbar{a} = (a_1, \dots, a_n) \in \mathbb{R}^n$. From \eqref{eq:FlVW} we deduce that
\[
{\mathcal  F}\big({\mathcal F}(x, \underbar a), \underbar b\big) = {\mathcal  F}(x, \underbar a + \underbar b), \qquad \forall x \in X, \, \forall \underbar a, \underbar b \in \mathbb R^n.
\]
Thus, ${\mathcal F}$ is a group action of $\mathbb R^n$ on $X$.  Furthermore, by Proposition \ref{prop-IsomMeasPres}, $\mathcal F$ acts by isometries, i.e.
\[
{\mathcal F}(\cdot, \underbar a) \colon X \to X \quad \text{is an isometry} \quad \forall \underbar a \in \mathbb R^n.
\]

Moreover, ${\mathcal F}$ acts transitively (see \cite[Proposition 4.6]{GR18}). That is, for any $x, y \in \X$ there exists a vector $\underbar a \in \R^n$ such that
\[
{\mathcal F}(x, \underbar a) = y \quad \text{ and } \quad |\underbar a | \le \de(x, y).
\]
This ensures that the stabilizer
of $\bar{x} \in \X$, defined as 
\[
\Gamma_{\bar x} := \Big\{ \underbar a \in \R^n : {\mathcal F}(\bar x, \underbar{a}) = \bar x \Big\} \subset \R^n,
\]
does not depend on $\bar x$, and thus we denote it simply by $\Gamma$. Clearly $\Gamma$ is a subgroup of $\R^n$ which is closed because ${\mathcal F}$ is a continuous function and it is discrete as shown in \cite[Proposition 4.7]{GR18}.

Fix some $\bar x \in X$. We then consider the quotient space $\R^n / \Gamma$, equipped with the only Riemannian metric letting the quotient map, $\R^n \longrightarrow \R^n / \Gamma$ be a Riemannian submersion. In particular, the distance induced by this metric is given by
\[
\de_{\R^n / \Gamma}\big([\underbar a], [\underbar b]\big) = \min_{\tiny{\begin{split}\underbar a' : [\underbar a'] = [\underbar a]\\ \underbar b' : [\underbar b'] = [\underbar b]\end{split}}}|\underbar a' - \underbar b'|,
\]
while the volume measure $\m_{\R^n / \Gamma}$ is, up to multiplication by a positive constant, the canonical Haar measure.

The map ${\mathcal F}(\bar x, \cdot): \R^n \to X$, passes to the quotient, inducing a map 
\[
\hat {\mathcal F} \colon \R^n/\Gamma \to \X, \qquad
\hat {\mathcal F}([\underbar a]) := {\mathcal F}(\bar x, \underbar a) \quad \forall \bar a \in \R^n.
\]
In particular, $\hat{\mathcal F}$ is an isometry and $\hat {\mathcal F}_\sharp \m_{\R^n / \Gamma} = c \m$ for some $c > 0$ (see \cite[Theorem 4.8]{GR18}). This fact together with \eqref{eq:mfinite} ensures that $\m_{\R^n/\Gamma}$ is finite. 

In order to conclude, recall that all non-trivial discrete subgroups of $\R^n$ are isomorphic to $\mathbb Z^d$ for $1 \leq d \le n$ and $\R^n/\mathbb Z^d$ has finite volume only when 
$d = n$. Thus, $\Gamma$ must be isomorphic to $\mathbb Z^n$. This implies that the quotient space $\R^n /\Gamma$ is a flat torus $\mathbb T^n$ and the quotient map $\hat {\mathcal F} \colon \mathbb T^n \to \X$ is the required isomorphism.
\end{proof}

\begin{remark}\label{bilip}
    For a lattice $\Gamma$ in $\mathbb{R}^n$, denote by $T_{\Gamma}:=\mathbb{R}^n/\Gamma$ the corresponding $n$-flat torus. It is easy to see that for two flat tori $T_{\Gamma_i}, i=1,2,$ satisfying $\vol_{T_{\Gamma_i}}(T_{\Gamma_i})\ge v>0$ and $\diam_{T_{\Gamma_i}}\le D$, if a group isomorphism $\phi:\Gamma_1 \to \Gamma_2$ is fixed, then there exists a canonical affine diffeomorphism, which is also a harmonic map 
    (see for instance \cite[Section 2]{Ham}),
    $F_{\phi}:T_{\Gamma_1} \to T_{\Gamma_2}$ induced by $\phi$ such that for all $x, y \in T_{\Gamma_1}$
    $$C_1\de_{T_{\Gamma_1}}(x, y) \le \de_{T_{\Gamma_2}}(F_{\phi}(x), F_{\phi}(y)) \le C_2\de_{T_{\Gamma_1}}(x, y)$$
    where $C_i=C_i(n, v, D)>0$. 
    
    This property will play a role in the proof of Theorem \ref{mthm0}. Finally it is worth mentioning that the space of all flat metrics on a $n$-torus with $\diam\le D$ and $\vol \ge v>0$ for fixed $D, v>0$ is compact with respect to the $C^{\infty}$-topology.
\end{remark}

\section{Proof of the main theorems}

This section is devoted to the proofs of Theorem \ref{mthm0} and Theorem \ref{th:main}. 

\subsection{Preliminary results}

We begin by establishing a consequence of Theorem \ref{thm-torus} for smooth manifolds.

\begin{theorem}\label{thm:new}(almost rigidity of Theorem \ref{thm-torus})
Let $n\in \N$ and $K \in \R, \delta, v, D>0$. 
If a closed $n$-dimensional Riemannian manifold ${M}$ satisfies $\ric_M \ge K$, $\vol_M(M) \ge v$, $\diam_M \le D$, and
\begin{equation}
\label{eq:assNew}
-\!\!\!\!\!\!\int_{M}g_M(V_i, V_j)\de\!\vol_M=\delta_{ij},\quad \int_{M}|\nabla V_i|^2 \de\!\vol_M \le \delta
\end{equation}
for some harmonic vector fields 
$V_1,\ldots, V_n$ on $M$, then $M$ is $\Psi$-Gromov-Hausdorff close to a flat $n$-torus $T$, where $\Psi=\Psi(\delta |n, K, v, D)$.
\end{theorem}

Our proof is by contradiction and relies on the structure of non-collapsed Ricci limit spaces. We thus recall some properties that we will use. Consider a non-collapsing sequence of closed smooth manifolds $M_i$ with $\ric_{M_i}\ge K$ converging to a compact $\RCD(K, n)$ space $(X, \de_X, \haus^n)$. The $n$-regular set $\mathcal{R}^n$ of $X$ has full measure and by \cite[Theorem 6.1]{ChCo}, we know that the singular set of $X$, $\mathcal{S}=X\setminus \mathcal{R}^n$, coincides with $\mathcal{S}^{n-2}$, where 
$$\mathcal S^{n-2}= \{x\in X: \mbox{no tangent cone at $x$ splits off $\R^{n-1}$}\}.$$

Now, a ball $B_{s}(x)\subset X$ is said to be $(n-1,\epsilon)$ symmetric if $d_{GH}(B_s(x), B_s(0))\leq \epsilon s$ for $B_s(0)\subset \R^n$ or $B_s(0)\subset [0,\infty)\times \R^{n-1}$. For $\epsilon, r>0$ we define 
$$\mathcal S^{n-2}_{\epsilon, r} =\{ x\in X: \mbox{there is no $s\in [r,1)$ such that $B_s(x)$ is $(n-1, \epsilon)$-symmetric}\},$$
and $\mathcal S^{n-2}_\epsilon= \bigcap_{r>0} \mathcal S^{n-2}_{\epsilon, r}$. It is well-known that $\mathcal{S}^{n-2}=\bigcup_{\epsilon>0}\mathcal{S}^{n-2}_{\epsilon}$. Moreover,  by  \cite[Theorem 1.9]{CJN} for any $\epsilon, r> 0$ the tubular neighbourhood of size $r$ of the set $\mathcal{S}^{n-2}_{\epsilon}$, denoted by $B_r(\mathcal{S}^{n-2}_{\epsilon})$, satisfies
\begin{equation}
\label{eq:sing_volume}
\haus^n\left(B_r(\mathcal{S}^{n-2}_{\epsilon})\right)\le C(\epsilon,n,v, D)r^{2},
\end{equation}
 where  $D$ is an upper bound of $\diam_X$ and $v$ is a positive lower bound of $\haus^n(X)$.

Now we recall the definition of 2-capacity of a set $\Omega\subset X$ (see for instance \cite[Chapters 1.4 and 6]{BB}), this is given by
\begin{align*}&\mathrm{Cap}_2(\Omega)=\\
&\ \ \ \ \ \ \inf\{\|u\|^2_{W^{1,2}(X)} \,:\, {u \in \mathrm{Lip}(X) \mbox{ such that }} u \geq 1 \mbox{ on a neighborhood of } \Omega \}.\end{align*}
Finding a Lipschitz function $0\le \phi_r \le 1$ satisfying $|\nabla \phi_r|\le \frac{3}{r}$, $\phi_r=1$ on $B_{\frac{r}{2}}(\mathcal{S}^{n-2}_{\epsilon})$, and $\supp \phi_r \subset B_r(\mathcal{S}^{n-2}_{\epsilon}))$ via the distance function from $\mathcal{S}^{n-2}_{\epsilon}$, by \eqref{eq:sing_volume} we have 
$$
\int_X|\nabla \phi_r|^2\de\!\haus^n\le C(\epsilon,n,v, D). $$
Then applying Mazur's lemma (see for instance \cite[Lemma 6.1]{BB}) and letting $r \to 0^+$, we have
$\mathrm{Cap}_2(\mathcal{S}^{n-2}_\epsilon)=0$ because we can find a $W^{1,2}$-strong limit of convex combinations of $\phi_r$ as $r \to 0^+$, where the limit must be $0$ by definition of $\phi_r$, and as a consequence $\mathrm{Cap}_2(\mathcal{S})=0$. Note that this argument still works without assuming the compactness of the spaces.

\begin{proof}
Assume that the conclusion does not hold. Thus, there exists
a non-collapsed measured Gromov-Hausdorff convergent sequence of $n$-dimensional closed Riemannian manifolds, 
\begin{equation*}
\left(M_i, g_{M_i}, \vol_{M_i}\right) \stackrel{\mathrm{mGH}}{\to} (X, \de_X, \haus^n)
\end{equation*}
satisfying 
\begin{equation*}
\ric_{M_i}\ge K,
\end{equation*}
for some $K \in \mathbb{R}$, and for any flat $n$-torus $T$ there exists  some $\tau>0$ such that
\begin{equation}\label{far}
\de_{\mathrm{GH}}(M_i, T) \ge \tau>0
\end{equation}
where the limit space $(X, \de_X)$  is compact {and $n$-dimensional. The compactness is a consequence of $\sup_i\diam_{M_i} <\infty$ and the $n$-dimensionality is a consequence of $\inf_i\vol_{M_i}>0$}.

{Moreover,} there exists a family of harmonic vector fields $V_{i,j}, j=1,\ldots, n$, on $M_i$ such that for all $i$, 
\begin{equation}\label{0}
-\!\!\!\!\!\!\int_{M_i}g_{M_i}(V_{i,j}, V_{i,k})\de\!\vol_{M_i}=\delta_{jk}, \qquad  \,\,\,\forall j\, \, \forall k,
\end{equation}
and for all $j$,
\begin{equation}\label{1}
\int_{M_i} |\nabla V_{i,j}|^2\de\!\vol_{M_i} \to 0 \quad \text{as $i \to \infty$}.
\end{equation}
\smallskip

{By  Theorems \ref{th:harmonicVF}, \ref{thm-Honda} and Lemma \ref{lem:div},} after passing to a subsequence, there exists a family $V_j$, $j=1,2,\ldots, n$, of $W^{1,2}_C$-vector fields on $X$ such that, as $i \to \infty$, $V_{i, j}$ $L^2$-strongly converge to $V_j$, $\nabla V_j =0$, $V_j \in D(\Div)$ and $\Div(V_j) =0$.
\smallskip

Since the $(1,2)$-Poincar\'e inequality on $M_i$ shows
\begin{align}\label{ppoincare}
&\int_{M_i}\left| g_{M_i}(V_{i,j}, V_{i,k})--\!\!\!\!\!\!
\int_{M_i}g_{M_i}(V_{i,j}, V_{i, k})
\de\!\vol_{M_i}  \right|  \de\!\vol_{M_i}  \nonumber \\
&\le C \int_{M_i}\left(|\nabla V_{i,j}|^2+|\nabla V_{i, k}|^2\right) \de\!\vol_{M_i},
\end{align}
letting $i \to \infty$, combined with (\ref{0}) and (\ref{1}), yields
\begin{equation}\label{2}
\langle V_j, V_k\rangle \equiv \delta_{jk} \ \ \ \haus^n\mbox{-almost everywhere.}
\end{equation} 

On the other hand, the arguments provided in the proof of Theorem \ref{th:harmonicVF} allow us to prove that for any $n$-regular point $x \in X$, there exist $r>0$ and $f_i \in D(\Delta, B_r(x))$ such that $V_i=\nabla f_i$ and $\Delta f_i=0$ hold. In particular, combining this with (\ref{2}) shows
\begin{equation*}
\langle \nabla f_j, \nabla f_k\rangle \equiv \delta_{jk}.
\end{equation*}

Since $\mathrm{Hess}_{f_i} \equiv 0$, we can apply the same argument as in the proof of the local splitting theorem \cite[Theorem 3.4]{BrueSemolaNaber_boundary} (even in this case where the Ricci curvature is bounded below by $K$), to conclude that the map $\Phi=(f_1, \ldots, f_n)$ gives an isometry from $B_s(x)$ to the ball in $\R^n$ of radius $s$ and centered at the origin, $B_s(0)$, for some small $0<s<r$. 
Hence, the $n$-dimensional regular set of $X$, $\mathcal{R}^n$, is open and locally flat. In particular it is smooth {after the identification between a small ball to a ball in $\mathbb{R}^n$ via $\Phi$ as above}. Therefore, by elliptic regularity theory, $V_j$ is smooth {(via the same identification)}. 
\smallskip

{\it Claim:}  $V_j \in H^{1,2}_C(T{\sf X})$.
\smallskip

Let us focus on the singular set $\mathcal{S}=X\setminus \mathcal{R}^n$. Since,
as we observed above,
\begin{equation*}
\mathrm{Cap}_2(\mathcal{S})=0,
\end{equation*}
we can find a sequence of smooth functions $\phi_i \in C^{\infty}_c(\mathcal{R}^n)$ that converges uniformly as $i \to \infty$ to $1$ on each compact subset of $\mathcal{R}^n$, satisfies $0\le \phi_i \le 1$, and has the property:
\begin{equation*}
\sup_i\int_{X}|\nabla \phi_i|^2\de\!\haus^n<\infty.
\end{equation*}
Consequently, $\phi_iV_j \in H^{1,2}_C(T{\sf X})$ (due to the smoothness of $V_j$) and
\begin{equation*}
\sup_i\|\phi_iV_j\|_{H^{1,2}_C}<\infty.
\end{equation*}
Letting $i \to \infty$ and applying Mazur's lemma yields $V_j \in H^{1,2}_C(T{\sf X})$.
This concludes the proof of the claim. 

To summarize, we have obtained $n$ vector fields $V_1, \ldots V_n$ which are $L^2$-orthogonal and satisfy $\nabla V_i=0$, $\mathrm{div} 
 (V_i)=0$ for all $i=1, \ldots , n$. Then the assumptions of Theorem \ref{thm-torus} are satisfied and $X$ is isomorphic to a flat torus. This contradicts (\ref{far}).
\end{proof}

\begin{remark}

In connection with Remark \ref{rem:WH}, we conjecture that for a non-collapsed $\RCD$ space, $W^{1,2}_C=H^{1,2}_C$ holds if and only if the space has no boundary. 
We note that there are two definitions of boundary in the literature \cite{kapovitch_mondino, DG}. Although it is not known whether these definitions coincide, it is known that for both notions the meaning of no boundary is the same due to \cite{BrueSemolaNaber_boundary}. 
\end{remark}

Note that under the assumptions of the previous theorem and its proof, we can prove that 
\begin{equation}\label{eq-small}
\left\|g_M-\sum_{j=1}^nV_j^{\sharp} \otimes V_j^{\sharp}\right\|_{L^p}
\end{equation}
is quantitatively small for any $1 \le p <\infty$. Moreover we deduce below that \eqref{eq-small} is small if and only if $M$ is quantitatively Gromov-Hausdorff close to a flat torus $T$.  The precise statement is as follows. Part (1) of Theorem \ref{partial answer} should be compared to Conjecture \ref{conjecture}.

\begin{theorem}\label{partial answer}
Let $K \in \mathbb{R}, n \in \mathbb{N}$ and $v, D>0$. There exists $\epsilon>0$ such that the following hold.
Let $M$ be a closed Riemannian manifold of dimension $n \ge 3$ with $\mathrm{ric}_M\ge K$, $\mathrm{vol}_M(M)\ge v$ and $\mathrm{diam}_M\le D$. Then the following hold.
\begin{enumerate}
\item If, 
{recalling that $r_M(x)$ is the smallest eigenvalue of $\mathrm{ric}_M$ at $x$,
\begin{equation}\label{smallscalar}
    \int_M (r_M)^-\mathrm{d}\mathrm{vol}_M\ge -\delta
\end{equation}
} holds for some $\delta \in (0, \epsilon],$  and 
\begin{equation}\label{l1 bound}
\left\| g_M-\sum_{j=1}^nV_j^{\sharp} \otimes V_j^{\sharp} \right\|_{L^1} \le \delta,
\end{equation}
for some harmonic vector fields $V_j$ on $M$, $j=1,2,\ldots, n$, then $M$ is $\Psi$-Gromov-Hausdorff close to a flat torus $T$, where $\Psi=\Psi(\delta|n, K, v, D)$.
\item If $M$ is $\delta$-Gromov-Hausdorff close to a flat torus $T$, then there exist harmonic vector fields $V_j$ on $M$, $j=1,2,\ldots, n$, such that for any $1\leq p < \infty$
\begin{equation}\label{l1 bound2}
\left\| g_M-\sum_{j=1}^nV_j^{\sharp} \otimes V_j^{\sharp} \right\|_{L^p} \le \Psi(\delta |n, K, v, D, p),
\end{equation}
Moreover, {if (\ref{smallscalar}) holds, then} any harmonic vector field $V$ on $M$ with $\|V\|_{L^2}\le 1$ satisfies
\begin{equation}\label{almost 0}
\int_M|\nabla V|^2\de\!\mathrm{vol}_M\le \Psi(\delta |n, K, v, D).
\end{equation}
\end{enumerate}
\end{theorem}

\begin{proof}
Let us prove (1):
Taking the trace of the left-hand side of (\ref{l1 bound}), we have $\|V_j\|_{L^2}\le C(n)$. Thus,  by elliptic regularity theory as observed in the 3rd step of the proof of Theorem \ref{th:harmonicVF}, we conclude $\|V_j\|_{L^{\infty}} \le C(n, K, D)$. Now, Bochner's formula yields
\begin{align*}
\frac{1}{2}\Delta_M |V_j|^2 &=|\nabla V_j|^2 +\mathrm{ric}_M(V_j, V_j) \\
&{\ge |\nabla V_j|^2+\langle r_Mg_M, V_j^{\sharp}\otimes V_j^{\sharp}\rangle }
\end{align*}
summing over $j$ gives
\begin{equation*}
\frac{1}{2}\Delta_M \left(\sum_{j=1}^n|V_j|^2\right) {\ge} \sum_{j=1}^n|\nabla V_j|^2 + \left\langle {r_Mg_M}, \sum_{j=1}^nV_j^{\sharp}\otimes V_j^{\sharp}-g_M\right\rangle +{nr_M}.
\end{equation*}
Integrating over $M$ yields
\begin{align}\label{l2zero}
\sum_{j=1}^n\int_M|\nabla V_j|^2 \mathrm{d}\mathrm{vol}_M &\le |K|{n} \int_M\left| g_M-\sum_{j=1}^nV_j^{\sharp} \otimes V_j^{\sharp} \right| \mathrm{d}\mathrm{vol}_M + {n}\int_M{(r_M)^-}\mathrm{d}\mathrm{vol}_M \nonumber \\
& \le |K|{n}\delta +{n}\delta .
\end{align}

To establish (1), we now follow the same arguments as in the proof of Theorem \ref{thm:new}. By a contradiction argument, we obtain a non-collapsed limit space $\sf X$ and 
limit parallel vector fields $V_j$ of harmonic (and almost parallel) vector fields  $V_{i,j}$ in the approximating sequence $M_i$. By (\ref{l1 bound}) applied to $V_{i,j}$ and the contradiction argument, $V_j$ satisfy, 
\begin{equation}\label{limit eq}
g_X=\sum_{j=1}^nV_j^{\sharp} \otimes V_j^{\sharp}.
\end{equation}
Denoting $A=(\langle V_j, V_k\rangle)_{jk}$, we have $A^2=A$, implying that any eigenvalue of $A$ is $0$ or $1$. If some eigenvalue of $A$ is $0$, it follows from (\ref{limit eq}) that the analytic dimension of $\sf X$ is less than $n$, contradicting the fact that $\sf X$ is non-collapsed. Thus, $A$ is non-degenerate, namely $\{V_j\}_j$ are linearly independent. Now, if necessary, we orthonormalize $\{V_j\}_j$ and then apply 
{Theorem \ref{thm-torus} with Remark \ref{rmk-torus}} to conclude that $\sf X$ is isomorphic to a flat torus.\\
 
Next we prove (2). The proof is done by contradiction. Thus, we consider a non-collapsed measured Gromov-Hausdorff convergent sequence of Riemannian manifolds $M_i$ of dimension $n$ with Ricci curvature bounded below uniformly, whose limit is a flat $n$-torus $T$,
\begin{equation}\label{noncollapsing torus}
\left(M_i, g_{M_i}, \mathrm{vol}_{M_i}\right) \stackrel{\mathrm{mGH}}{\to} \left(T, \de_T, \haus^n\right).
\end{equation}
Then, it is enough to prove that there exists a sequence of harmonic vector fields $V_{i, j}$ on $M_i$, $j=1,2,\ldots, n$, such that for any $1\le p <\infty$ we have
\begin{equation}\label{app Riem}
    \left\| g_{M_i}-\sum_{j=1}^nV_{i,j}^{\sharp}\otimes V_{i, j}^{\sharp}\right\|_{L^p} \to 0,\quad \text{as $i \to \infty$.}
\end{equation}
To do so, take harmonic vector fields $V_j$ on $T$, $j=1,2,\ldots,n$, with $g_T=\sum_{j=1}^nV_j^{\sharp}\otimes V_j^{\sharp}$. Then let us recall;
\begin{itemize}
    \item the spectral convergence of the Hodge Laplacian $\Delta_{H, 1}$ acting on $1$-forms holds in the setting (\ref{noncollapsing torus})
    as established by  \cite[Theorem 1.1]{Honda17}, along with the discussions in \cite[Section 5]{Honda17} (namely the spectral convergence of $\Delta_{H, 1}$ holds if the non-collapsed limit space satisfies $H^{1,2}_C=W^{1,2}_C$). 
    \item For any sufficiently large $i$, by \cite[Theorem A.1.12]{ChCo} (or see Proposition \ref{approx}), $M_i$ is diffeomorphic to $T$; in particular $b_1(M_i) = b_1(T) = n$.
\end{itemize}
Thus, each $V_j$ can be $L^2$-strongly approximated by a sequence of harmonic vectors fields $V_{i, j}$ on $M_i$. Hence, we obtain (\ref{app Riem}) for the case when $p=1$. The extension to the general case $1\leq p<\infty$ is given by the case $p=1$ and the quantitative $L^{\infty}$ estimate on $V_{i, j}$ as observed in step 3 of the proof of Theorem \ref{th:harmonicVF}.

The final statement, (\ref{almost 0}), follows from (\ref{l2zero}) and the fact that any harmonic vector field on $M_i$ is spanned by $V_{i, j}$ since $M_i$ is diffeomorphic to $T$.
\end{proof}

We end this subsection by stating the following proposition which will be used in the proof of Theorem \ref{mthm0}. Note that we also use this proposition to prove Theorem \ref{th:main}.

\begin{proposition}\label{prop-lower b}
    For all $n \in \mathbb{N}, K \in \mathbb{R}, c, D>0$ and $L>0$ there exist $\delta=\delta(n, K, c, D, L)>0$ and $\tau=\tau(n, K, c, D, L)>0$ such that the following holds. Let $M$ be a closed Riemannian manifold of dimension $n$ with $\ric_M \ge K$ and $\diam_M\le D$, and let
    \begin{equation*}
    \Phi=(\phi_i)_{i=1}^n:M \to \mathbb{R}^n/\mathbb{Z}^n
    \end{equation*}
    be a harmonic map with $E(\Phi)\le L$, $\haus^n(\Phi(M)) \ge c$ {(in particular they imply $\vol_M(M) \ge v(n, K, c, D, L)>0$ by the Lipschitz continuity of $\Phi$)} and 
    $$\int_M|\nabla d\phi_i|^2\de\!\vol_M\le \delta$$ 
    for any $i=1,2,\ldots, n$. Then $D(\Phi) \ge \tau$.
    \end{proposition}

\begin{proof}
    The proof is done by contradiction. Then our setting is: 
    \begin{itemize}
    \item let
    \begin{equation*}
        (M_i, g_{M_i}, \vol_{M_i}) \stackrel{\mathrm{mGH}}{\to} (X, \de_X, \haus^n)
    \end{equation*}
    be a non-collapsed measured Gromov-Hausdorff convergent sequence of Riemannian manifolds $M_i$ with $\ric_{M_i}\ge K, \diam_{M_i}\le D$ and $\vol_{M_i}(M_i)\ge v$;
    \item let $$\Phi_i=(\phi_{i,j})_{j=1}^n:M_i \to \mathbb{R}^n/\mathbb{Z}^n$$ be a  harmonic map with $E(\Phi_i) \le L$, $\haus^n(\Phi_i(M_i))\ge c$,
    \begin{equation}\label{deg}
        D(\Phi_i) \to 0, \quad \int_{M_i}|\nabla d\phi_{i,j}|^2\de\!\vol_{M_i} \to 0.
    \end{equation}
    \end{itemize}
    After passing to a subsequence, we can find the uniform limit harmonic map $\Phi=(\phi_i)_{i=1}^n:X \to \mathbb{R}^n/\mathbb{Z}^n$ of $\Phi_i$. In particular, we have 
    \begin{equation}\label{positive meas}
        \haus^n(\Phi(X))\ge \limsup_{i\to \infty}\haus^n(\Phi_i(X_i))\ge c.
    \end{equation}

    On the other hand, (\ref{deg}) with the Poincar\'e inequality implies that $\langle d\phi_j, d\phi_k\rangle$ is constant with $\mathrm{det}(\langle d\phi_j, d\phi_k\rangle )_{jk}=0$ like in \eqref{ppoincare}. Thus, the proof of the local splitting theorem \cite[Theorem 3.4]{BrueSemolaNaber_boundary} allows us to conclude that for any $x \in X$, there exists $r>0$ such that after regarding the map $\Phi$ on $B_r(x)$ as $\Phi:B_{r}(x) \to \mathbb{R}^n$, $\Phi$ can be locally
    identified with the projection to $\mathbb{R}^k$ for some $k <n$. In particular, the Hausdorff dimension of the image $\Phi(X)$ is at most $n-1$. This contradicts (\ref{positive meas}).
\end{proof}

\subsection{Proof of Theorem \ref{mthm0}}
We are now in a position to prove Theorem \ref{mthm0}.

Let us recall the setting of this theorem: Let $M$ be a closed Riemannian manifold of dimension $n$ with $\diam_M\le D$ and $\ric_M\ge K$ for some $D>0$ and some $K \in \mathbb{R}$.

Firstly, let's prove (1), the statement of which is recalled as follows: if there exists a $(\delta;L, \tau)$-harmonic map $$\Phi=(\phi_i)_{i=1}^n:M \to \mathbb{R}^n/\mathbb{Z}^n,$$ 
then $M$ is $\Psi$-Gromov-Hausdorff close to a flat $n$-torus $T$,
for some $\Psi=\Psi(\delta| n, K, D, L, \tau)$.

Furthermore, for $\delta$ small enough depending only on $n, K, D, L$ and $\tau$, we have the following.
\begin{itemize}
    \item[(a)] $M$ is diffeomorphic to $T$;
    \item[(b)] $\Phi$ is a covering map with $\mathrm{deg}(\Phi) \le C_0$; in particular $\Phi$ is surjective. More strongly, there exists an affine harmonic diffeomorphism $H:\mathbb{R}^n/\mathbb{Z}^n \to T$ such that $H$ is $C_1$-bi-Lipschitz (namely both $H, H^{-1}$ are $C_1$-Lipschitz) and that for all $x, y \in M$ with $\de_M(x, y) \le r$,
\begin{equation}\label{holderlip}
(1-\Psi)\de_M(x, y)^{1+\Psi} \le \de_{T}(H\circ \Phi(x), H \circ \Phi(y)) \le C_2\de_M(x, y),
\end{equation}
 where $C_i=C_i(K, n, D, L, \tau)>0$ and $r=r(K, n, D, L, \tau)>0$.
    \item[(c)] if $\mathrm{deg}(\Phi)=1$, then $\Phi$ is a diffeomorphism, (\ref{holderlip}) is satisfied for all $x, y \in M$, and $H \circ \Phi$ is a $\Psi$-Gromov-Hausdorff approximation.
\end{itemize} 

\begin{proof}[Proof of (1) of Theorem \ref{mthm0}]
The proof is done by contradiction.
Thus it is enough to consider the following setting: let 
\begin{equation}\label{mghconv}
\left(M_i, g_{M_i}, \frac{\mathrm{vol}_{M_i}}{\mathrm{vol}_{M_i}(M_i)}\right) \stackrel{\mathrm{mGH}}{\to} {\sf X} =(X, \de_X, \m_X)
\end{equation}
be a measured Gromov-Hausdorff convergent sequence of {closed} Riemannian manifolds $M_i$ of dimension $n$ with $\ric_{M_i}\ge K$ {and $\diam_{M_i}\le D$}, to a compact $\RCD(K, n)$ space $\sf X$ and let 
$$\Phi_i=(\phi_{i,j})_{j=1}^n:M_i \to \mathbb{R}^n/\mathbb{Z}^n$$
be a $(\delta_i;C, \tau)$-harmonic map for some $\delta_i \to 0^+, C>0$ and some $\tau>0$. Then the following arguments allow us to get all conclusions stated in (1) via a contradiction. 

Firstly let us prove that $(X, \de_X)$ is isometric to a flat $n$-torus. Thanks to (\ref{EL}) with the stability of Laplacian \cite[Theorem 4.4]{ambrosio_honda_local}, after passing to a subsequence, there exists the $W^{1,2}$-strong limit harmonic map $\Phi=(\phi_j)_{j=1}^n:X \to \mathbb{R}^n/\mathbb{Z}^n$ of $\Phi_{i}$ (see Subsection \ref{circleSob}, in particular (\ref{EL})). Thus 
\begin{equation}\label{positivedef}
D(\Phi)=\lim_{i \to \infty}D(\Phi_i)\geq \tau>0,
\end{equation}
namely $\{d\phi_j\}_{j=1}^n$ are linearly independent in $L^2(T{\sf X})$. Moreover Theorem \ref{thm-Honda} allows us to conclude $\nabla d\phi_j=0$. Thus combining this with (\ref{positivedef}) shows
$$
\det \left(\langle d\phi_j, d\phi_k\rangle \right)_{jk}(x)>0
$$
for $\m_X$-a.e. $x \in X$. Therefore we see that the analytic dimension of $X$ is at least $n$, thus the limit space $X$ is non-collapsed.
Then we can apply the same argument as in the proof of Theorem \ref{thm:new} for the duals of $\{d\phi_{i,j}\}_{j=1}^n$
and conclude that $X$ is isomorphic to a flat $n$-torus.

Secondly let us prove that $\Phi_i$ gives a local diffeomorphism for any sufficiently large $i$. Denote by $\Phi=(\phi_j)_{j=1}^n$ the limit harmonic map of $\Phi_i$ discussed above. Thanks to the equi-Lipschitz continuity of $\Phi_i$ (recall Part \textbf{3} of the proof of Theorem \ref{th:harmonicVF}), we can take a small $r>0$ satisfying that the restriction of $\Phi_i$ to any ball of radius at most $r$ can be regarded as a map into $\mathbb{R}^n$ because of the flatness of $X$. In the rest of the proof, let us work under this convention; $$\Phi_i:B_r(x_i) \to \mathbb{R}^n,\quad \Phi:B_r(x) \to \mathbb{R}^n,$$
with a fixed small $r>0$, where $x_i \in M_i$ converge to $x \in X$. Note that $r$ can be taken quantitatively.

Denoting by $L_B:\mathbb{R}^n \to \mathbb{R}^n$ the linear map defined by multiplying a matrix $B$ of size $n$ (namely $L_B(v)=Bv$),
note that $L_A \circ \Phi :B_r(x) \to \mathbb{R}^n$ for some $A \in GL(n;\mathbb{R})$ with $|A|\le C(n, D, v, \tau)$ and $\det (A) \ge c(n, D, v, \tau)>0$ is an isometry onto its image, where $D$ is an upper bound on 
$\diam_X$, $\tau$ is a positive lower bound on $D(\Phi)$, and $v$ is a positive lower bound on $\vol_X(X)$ (see also Remark \ref{bilip}).  Thus 
applying the same arguments in \cite[Section 4]{HondaPeng} for regular maps $L_A \circ \Phi_i$ (see also Proposition \ref{approx}), we know
\begin{equation}\label{local scale}
(1-\delta_i)\de_{M_i}(y_i, z_i)^{1+\delta_i} \le |\Phi_i(y_i)-\Phi_i(z_i)|_{\mathbb{R}^n} \le C(n)\de_{M_i}(y_i, z_i)
\end{equation}
for all $y_i, z_i \in B_r(x_i)$, after replacing $r$ by a (quantitatively) smaller one if necessary. In particular $\Phi_i$ is non-degenerate, thus we conclude that $\Phi_i$ is a local diffeomorphism. 

Thirdly let us prove that $\Phi_i$ is a covering map. By (\ref{local scale}) it is enough to check 
the surjectivity of $\Phi_i$. However it is a direct consequence of invariance of domain. Namely, (\ref{local scale}) with invariance of domain implies that $\Phi_i(M_i)$ is open, moreover, $\Phi_i(M_i)$ is compact, in particular closed in $\mathbb{R}^n/\mathbb{Z}^n$ because of the compactness of $M_i$, thus $\Phi_i(X_i)=\mathbb{R}^n/\mathbb{Z}^n$. 

Finally let us prove the remaining statements of (1). It follows from (\ref{local scale}) that for any $x \in \mathbb{R}^n/\mathbb{Z}^n$, taking $y \in \Phi_i^{-1}(x)$ and for $r>0$ small enough, we have
$$|\mathrm{deg}(\Phi_i)| \le \sharp \Phi_i^{-1}(x) \le \frac{\haus^n(X_i)}{\haus^n(B_r(y))}.$$
Since we have a quantitative upper bound of the right-hand-side because of the Bishop-Gromov inequality, we obtain the desired degree estimate.

Combining the above observations with Proposition \ref{approx}, Remark \ref{bilip} (about the existence of harmonic maps between flat tori) and the standard topological arguments, we easily obtain the remaining statements about degree $1$. 
Namely, we assume $\mathrm{deg}(\Phi_i)=1$ for any sufficiently large $i$ in the sequel. Then $\Phi_i$ must be injective because $\Phi_i$ is a covering map. Therefore $\Phi_i$ gives a diffeomorphism. 

Next let us prove that $\Phi$ is also a diffeomorphism. Since the surjectivity of $\Phi$ comes from that of $\Phi_i$, it is enough to check the injectivity of $\Phi$. 
Since $\Phi$ is also a covering map (because of applying the results above to $\Phi$), it is also enough to prove $\mathrm{deg}(\Phi)=1$. However this can be easily checked by the fact that $(\Phi_i)_*\vol_{M_i}$ weakly converge to $\Phi_*\haus^n$, which is a direct consequence of the uniform convergence of $\Phi_i$ to $\Phi$. 
More precisely, since $((\Phi_i)_*\vol_{M_i})(\mathbb{R}^n/\mathbb{Z}^n) \to (\Phi_*\haus^n)(\mathbb{R}^n/\mathbb{Z}^n)$, by co-area formula, this can be written as
$$\int_{M_i}\sqrt{D(\Phi_i)}\de\!\vol_{M_i} \to \mathrm{deg}(\Phi) \int_X\sqrt{D(\Phi)}\de\!\haus^n.$$
On the other hand, the left-hand-side of the above must converge to $\int_X\sqrt{D(\Phi)}\de\!\haus^n$ because $\sqrt{D(\Phi_i)}$ $L^p$-strongly converge to $\sqrt{D(\Phi)}$ for any $1 \le p<\infty$. This shows $\mathrm{deg}(\Phi)=1$.
Thus $\Phi$ gives a diffeomorphism.

Then consider a diffeomorphism $\bar \Phi_i:=\Phi^{-1} \circ \Phi_i:M_i \to X$ and observe that $\bar \Phi_i$ converge uniformly to the identity with respect to (\ref{mghconv}). In particular $\bar \Phi_i$ gives an $\epsilon_i$-Gromov-Hausdorff approximation for some $\epsilon_i \to 0^+$.
Thus noticing that $\bar \Phi_i$ are equi-regular, we can apply Proposition \ref{approx} to conclude that for some $\epsilon_i \to 0^+$, we have
\begin{equation*}
(1-\epsilon_i)\de_{M_i}(x, y)^{1+\epsilon_i} \le \de_{X}(\bar \Phi_i(x),  \bar \Phi_i(y)) \le C(n)\de_{M_i}(x, y)
\end{equation*}
for all $x, y \in M_i$.
Then we complete the proof.

\end{proof}

Finally let us provide a proof of (2) of Theorem \ref{mthm0} whose statement is as follows: if $M$ is $\delta$-Gromov-Hausdorff close to a flat $n$-torus $T$ with $\mathrm{vol}_{T}(T) \ge v$ for some $v>0$, 
then there exists a  $(\Psi;C_3, C_4)$-harmonic diffeomorphism such that
\begin{equation*}
\Phi=(\phi_j)_{j=1}^n:M\to \mathbb{R}^n/\mathbb{Z}^n,
\end{equation*}
where $C_i=C_i(K,n, D, v)>0$ and $\Psi=\Psi(\delta|n, K, D, v)>0$.

\begin{proof}[Proof of (2) of Theorem \ref{mthm0}]
The proof is done by contradiction. Thus by (1) it is enough to prove that for a  non-collapsed measured Gromov-Hausdorff convergent sequence of manifolds with Ricci curvature bounded below  to a flat torus $X$ with $\mathrm{diam}_X\le D$ and $\mathrm{vol}_X(X) \ge v$;
$$(M_i, g_{M_i}, \mathrm{vol}_{M_i}) \stackrel{\mathrm{mGH}}{\to} (X, g_{X}, \mathrm{vol}_{X}),$$
the existence of harmonic diffeomorphisms
$$\Phi_i=(\phi_{i, j})_{j=1}^n:M_i \to \mathbb{R}^n/\mathbb{Z}^n$$
with $\liminf_iD(\Phi_i)>0$.

Take a diffeomorphism $F_i:M_i \to X$ for any sufficiently large $i$ as stated in Remark \ref{approxrem}, and fix an affine harmonic diffeomorphism $\Phi:X \to \mathbb{R}^n/\mathbb{Z}^n$ as explained in Remark \ref{bilip}.

Then let us find the harmonic map representative $\Phi_i=(\phi_{i,j})_{j=1}^n$ in the homotopy class of $\Phi \circ F_i$ as done in \cite{ES} (see also \cite{Hart}) by applying the  harmonic map heat flow $h_t(\Phi \circ F_i)$ as $t \to \infty$. Note that the energies are not increasing along the flow (see for instance the proposition in the bottom of page 135 in \cite{ES}) with 
\begin{align*}
    E(\Phi_i)&  \leq  \frac{1}{2}\int_{M_i}\langle (\Phi \circ F_i)^*g_{\R^n/\Z^n}, g_{M_i}\rangle \de\!\vol_{M_i} \\
    &\le C(n, D, v) \int_{M_i}\langle F_i^*g_{X}, g_{M_i}\rangle \de\!\vol_{M_i} \\
    &=C(n, D, v) \int_{M_i}\langle F_i^*g_{X}-g_{M_i}, g_{M_i}\rangle \de\!\vol_{M_i}+C(n, D, v)\int_{M_i}|g_{M_i}|^2\de\!\vol_{M_i}\\
    &\le C(n, D, v) \int_{M_i}|F_i^*g_{X}-g_{M_i}|\de\!\vol_{M_i}+C(n, D, v) \vol_{M_i}(M_i)n  
    \\
    &\to C(n, D, v) \vol_{X}(X)n,
\end{align*} 
because of Proposition \ref{approx}.
Since $\Phi_i$ is homotopic to the diffeomorphism $\Phi \circ F_i$, the degree of $\Phi_i$ is equal to $1$. Therefore $\Phi_i$ must be surjective. Therefore Proposition \ref{prop-lower b} with (2) of Theorem \ref{partial answer}  
yields
\begin{equation*}
D(\Phi_i)\ge \tau(n, K, D, v)>0
\end{equation*}
which completes the proof. \end{proof}

\subsection{Proof of  Theorem \ref{th:main}}

We now proceed to the proof of Theorem \ref{th:main}. 
Firstly we recall the following result of Stern \cite[Theorem 1.1]{stern}. 

\begin{theorem}[\cite{stern}]
Let $M$ be a closed, oriented Riemannian 3-manifold, and let $u: M \to \mathbb{R}/\mathbb{Z} \cong \mathbb{S}^1(\frac{1}{2\pi})$ be a harmonic map such that $du$ does not vanish in $H^1_{\mathrm{dR}}(M,\R)$. Then, 
\begin{align}\label{ineq:stern}
2\pi \int_{\mathbb{R}/\mathbb{Z}} \chi(\Sigma_{\theta}) \de\!\theta\geq \frac{1}{2} \int_{\mathbb{R}/\mathbb{Z}} \int_{\Sigma_{ \theta}}\left(\frac{|\nabla du|^2}{|du|^2}  + R_M\right) { \de\!\vol_{\Sigma_{\theta}}} \de\!\theta, 
\end{align}
{where 
$\Sigma_{\theta}:=u^{-1}(\theta)$ is a regular surface for a.e. $\theta \in \mathbb{R}/\mathbb{Z}$ and $\chi(\Sigma_{\theta})$ is its Euler characteristic.}
\end{theorem}

\begin{corollary}\label{cor:stern}
Let $M$ be a Riemannian 3-torus 
with $\ric_M\geq K$, for some $K\in \R$, and $\diam_M\leq D$, for some $D>0$. 
Then for any $L >0$ there exists a constant $C=C(K, D, L)>0$ such that any harmonic map $u:M \to \mathbb{R}/\mathbb{Z}$ with $\|du\|_{L^2}\le L$ satisfies
\[
0\geq \int_{M} {|\nabla du|^2}\de\!\vol_M - C\int_{M} R^-_M \de\!\vol_M,
\]
where recall $R^-_M= \max\{-R_M, 0\}.$
\end{corollary}

\begin{proof} We apply \eqref{ineq:stern} with $du$. Since $M$ is a torus, it does not contain non-separating spheres, the left hand side  of \eqref{ineq:stern} is non positive.
Hence in combination with  the co-area formula it follows
\[0\geq \int_M \frac{|\nabla du|^2}{|du|} \de\!\vol_M - \int_{M} R^-_{M}|du| \de\!\vol_M.\]
As already observed in Part {\bf 3} of the proof of Theorem  \ref{th:harmonicVF}, there exists a constant $C=C(K, D, L)$ such that ${\|du\|_{L^{\infty}}}\leq C$.  Therefore
$$0\geq \int_{M} |\nabla du|^2 \de\!\vol_M -C \int_{M} R^-_M \de\!\vol_M.\eqno \qedhere$$
\end{proof}
We are now in position to give the proof of Theorem \ref{th:main}.

\begin{proof}[Proof of Theorem \ref{th:main}]
Let us divide the proof into 3 steps as follows.

\textbf{Step 1}: We prove the existence of a smooth diffeomorphism $F:M \to \mathbb{R}^3/\mathbb{Z}^3$ that is $C(K, D, v)$-Lipschitz. In particular, $E(F) \le C(K, D, v)$.

The proof is based on Ricci flow smoothing.
Fix $D, v>0$, and consider a family of positive numbers $\{c(k)\}_k$.
Let $\mathcal{M}=\mathcal{M}(D, v, \{c(k)\}_k)$ be the moduli space of Riemannian metrics $g$ on $\mathbb{R}^3/\mathbb{Z}^3$ satisfying $\diam_g\le D$, $|\nabla^k\mathrm{Riem}_g| \le c(k)$ and $\vol_g(M) \ge v$, where $\mathrm{Riem}_g$ denotes the Riemannian curvature tensor.  We know that $\mathcal{M}$ is compact with respect to the $C^{\infty}$-topology.
In particular, there exists $C=C(D, v, \{c(k)\}_k)>0$ such that the identity map from $(\mathbb{R}^3/\mathbb{Z}^3, g)$ to $(\mathbb{R}^3/\mathbb{Z}^3, g_{\mathbb{R}^3/\mathbb{Z}^3})$ is $C$-Lipschitz for any $g \in \mathcal{M}$.

Let us denote by $g_M$ the Riemannian metric on $M$.
On the other hand, thanks to \cite[Theorem 1.1]{ST}, there exists $t_0=t_0(K, D, v)>0$ such that the Ricci flow $g_{t_0}$ at time $t_0$ starting from $g_M$ satisfies;
\begin{enumerate}
\item the identity map from $(M, \de_{g_M})$ to $(M, \de_{g_{t_0}})$ is $C(K, D, v)$-Lipschitz;
\item $|\nabla^k\mathrm{Riem}_{g_{t_0}}|\le C(K, k, D, v)$ for any $k \in \mathbb{Z}_{\ge 0}$;
\item $\ric_{g_{t_0}}\ge -\alpha(K, D, v)$ and $\vol_{g_{t_0}}(M)\ge \beta(K, D, v)>0$.
\end{enumerate}
Thus, due to the observation above and (2)-(3), we have a smooth diffeomorphism $\tilde F:(M, g_{t_0}) \to \mathbb{R}^3/\mathbb{Z}^3$ with $E_{g_{t_0}}(\tilde F) \le C(K, D, v)$ 
Moreover, by (1) above, the composition $$F:(M, \de_{g_M}) \stackrel{\mathrm{id}}{\to}(M, \de_{g_{t_0}}) \stackrel{\tilde F}{\to} \mathbb{R}^3/\mathbb{Z}^3$$ 
provides the desired map. Thus we have \textbf{Step 1}.

\textbf{Step 2}: There exists a surjective harmonic map $\Phi:M \to \mathbb{R}^3/\mathbb{Z}^3$ with $E(\Phi) \le C(K, D, v)$.

This is done by applying the harmonic map heat flow $h_t F$, as in the proof of Theorem \ref{mthm0}. Specifically,  the limit harmonic map $\Phi$ of $h_tF$ as $t \to \infty$ provides the desired harmonic map because $\Phi$ has degree $1$, and its energy is at most $E(F)$. 

\textbf{Step 3}: We can now conclude the proof. Let $u_i$, $i=1,2,3$ be the components of $\Phi$. Then by applying Corollary \ref{cor:stern} to each $u_i$ and using our assumption on $R_M^-$ we get
$$-\!\!\!\!\!\!\int_M |\nabla du_i|^2\de\!\vol_M \le C\delta.$$
We choose $\delta$ and $\tau$ so that we can apply Proposition \ref{prop-lower b}. Therefore we get that $\Phi$ is a $(\Psi; C_1, C_2)$-harmonic map, where $\Psi=\Psi(\delta|K, D, v)$ and $C_i=C_i(K, D, v)>0$. Then the conclusion follows from (1) in Theorem \ref{mthm0}. \end{proof}

Finally let us provide an analogous result of \cite[Theorem 1.11]{ABK} in our setting. Observe that we replace the $L^3$-bound on Ricci by a lower bound on Ricci and an $L^1$-bound on the negative part of the scalar curvature.
\begin{theorem}\label{th:ABK}
Let $M$ be a Riemannian manifold of dimension $3$ with $\ric_M \ge K, \diam_M\le D$ and $\vol_M(M) \ge v$ for some $K \in \mathbb{R}, D, v>0$. Assume the following:
\begin{enumerate}
\item the second integral homology $H_2(M;\mathbb{Z})$ contains no spherical classes;
\item there are $3$-cohomology classes $\alpha_i \in H^1_{\mathrm{dR}}(M;\mathbb{Z}), i=1,2,3$ such that $\alpha_1\wedge \alpha_2 \wedge \alpha_3$ generates $H^3_{\mathrm{dR}}(M;\mathbb{Z})$;
\item we have $\int_MR_M^-\de\!\vol_M \le \delta$ for some $\delta>0$.
\end{enumerate}
Then $M$ is $\Psi$-Gromov-Hausdorff close to a flat $3$-torus, where $\Psi=\Psi(\delta|K, D, v)$.
\end{theorem}
\begin{proof} 
Since the proof is very close to that of Theorem \ref{th:main} with technical results obtained in \cite{ABK}, let us provide only an outline of the proof. 

Applying \cite[Theorem 1.1]{ST} again, we can find a Ricci flow $g_{t_0}$ at time $t_0$ starting at $g_M$ such that $g_{t_0}$ is in a bounded geometry. In particular $(M, g_{t_0})$ is in a class introduced in \cite[Definition 1.9]{ABK}. Therefore it follows from \cite[Corollary 7.2]{ABK} (or applying a smooth compactness result based on curvature controls) that 
\begin{equation}\label{energy}
    \|\alpha_i\|_{L^2(g_{t_0})}\le C(K, D, v)
\end{equation}
holds for any $i=1,2,3$. Then we define a smooth map $\tilde F:(M, g_{t_0}) \to \mathbb{R}^3/\mathbb{Z}^3$ by
\begin{equation}
    \tilde F(x):=\left(\int_{\gamma_x}\alpha_i \right)_{i=1,2,3},
\end{equation}
where $\gamma_x$ is a smooth curve from a fixed point $x_0 \in M$ to $x \in M$. Note that the degree of $\tilde F$ does not vanish because
denoting by $\omega_{\mathbb{R}^3/\mathbb{Z}^3}$ the canonical volume $3$-form of $\mathbb{R}^3/\mathbb{Z}^3$, it follows from the definition of $\tilde F$ that 
$$\tilde F^*\omega_{\mathbb{R}^3/\mathbb{Z}^3}=\alpha_1\wedge \alpha_2\wedge \alpha_3$$
holds. In particular ${\tilde F^*}:H^3_{\mathrm{dR}}(\mathbb{R}^3/\mathbb{Z}^3;\mathbb{R})\to H^3_{\mathrm{dR}}(M;\mathbb{R})$ is injective, thus the degree of $\tilde F$ does not vanish.

On the other hand, $\tilde F$ must be surjective because
if $\tilde F$ is not surjective, then the induced map $(\tilde F)_*:H_2(M;\mathbb{Z}) \to H_2(T;\mathbb{Z})$ is obtained by the composition $H_2(M;\mathbb{Z})\to H_2(T \setminus B;\mathbb{Z}) \to H_2(T;\mathbb{Z})$ for some small open ball $B \subset T$, and $H_2(T\setminus B;\mathbb{Z})$ vanishes since $T\setminus B$ is homotopic to a bouquet, thus $\mathrm{deg}(\tilde F)=0$ which is a contradiction.

Therefore $\tilde F$ is surjective with $E_{g_{t_0}}(\tilde F) \le C(K, D, v)$ because of (\ref{energy}). Then considering the composition 
$$F:(M, \de_{g_M}) \stackrel{\mathrm{id}}{\to}(M, \de_{g_{t_0}}) \stackrel{\tilde F}{\to} \mathbb{R}^3/\mathbb{Z}^3,$$
we obtain the desired conclusion by similar arguments as in the proof of Theorem \ref{th:main}. \end{proof}
{
\subsection{Proof of Theorem \ref{thm:almostnonricci}.}
Finally let us prove the remaining result.

\begin{proof}[Proof of Theorem \ref{thm:almostnonricci}.]
Let us firstly prove (1). For the harmonic map $\Phi$, 
Proposition \ref{prop-lower b} with Remark \ref{rem:delta} implies that $\Phi$ is $(2L\delta;L, \tau)$-harmonic if $\delta$ is small depending only on $n, D, c$ and $L$, where $\tau=\tau(n, D, c, L)>0$.
Thus we can apply Theorem \ref{mthm0} to complete the proof of (1).

Finally let us prove (2).
As done in the proof of (2) of Theorem \ref{mthm0}, applying the harmonic map heat flow starting at $F$, we can find a harmonic map $\tilde \Phi:M \to \mathbb{R}^n/\mathbb{Z}^n$ whose degree is the same to that of $F$, with $E(\tilde \Phi)\le E(F) \le L$. Thus $\tilde \Phi$ must be surjective. In particular applying (1) for $\tilde \Phi$ completes the proof of (2).
\end{proof}
}
\small{
\bibliographystyle{amsalpha}}
\bibliography{scalar}
\end{document}